\documentclass[11pt]{article}
\usepackage{amsmath, amsthm, amssymb}
\usepackage{tkz-graph, subfig}
\usepackage{marginnote}
\usepackage{vwcol}
\usepackage{verbatim}
\usepackage[top=1.0in, bottom=1.0in, left=1.0in, right=1.0in]{geometry}
\usepackage{color, textcomp}
\usepackage[margin=.5in,size=footnotesize]{caption}
\tikzstyle{VertexStyle} = []
\tikzstyle{EdgeStyle} = [line width=1.2pt]
\tikzstyle{labeledStyle}=[shape = circle, minimum size = 6pt, inner sep = 0pt, outer sep = 0pt]
\tikzstyle{unlabeledStyle2}=[line width = 2.0pt, shape = circle, minimum size = 5pt, inner sep = 2pt, outer sep = 0pt, draw]
\tikzstyle{unlabeledStyle}=[shape = circle, minimum size = 5pt, inner sep = 2pt, outer sep = 0pt, draw]

\usepackage[]{hyperref}

\usepackage{sectsty}
\allsectionsfont{\sffamily}

\newcommand\core{{\textrm{core}}}
\newcommand\pot{{\textrm{pot}}}

\newcommand\none{\,\bowtie}

\theoremstyle{plain}
\newtheorem{thm}{Theorem}
\newtheorem{prop}[thm]{Proposition}
\newtheorem{lem}[thm]{Lemma}
\newtheorem{conj}{Conjecture}
\newtheorem{cor}[thm]{Corollary}

\newtheorem*{StrongMinor}{Strong Minor Lemma}
\newtheorem*{main-thm}{Main Theorem}
\newtheorem{clm}{Claim}

\theoremstyle{definition}
\newtheorem{defn}{Definition}
\theoremstyle{remark}

\newcommand{\fancy}[1]{\mathcal{#1}}

\newcommand\B{\fancy{B}}
\newcommand{\G}{\mathcal{G}}

\newcommand{\Int}{\textrm{Int}}

\newcommand{\card}[1]{\left|#1\right|}

\newcommand\Ggood{\G_{good}}
\newcommand\Gbad{\G_{bad}}
\newcommand\Gmixed{{K_{3,3}-e}}
\newcommand\Gcycles{\G_{cycles}}

\newcommand{\vph}{\varphi}
\newcommand{\sqed}{$\hfill\diamond$}

\newcommand{\erdos}{Erd\H{o}s}

\newcommand{\aside}[1]{\marginnote{\scriptsize{#1}}[0cm]}
\newcommand{\aaside}[2]{\marginnote{\scriptsize{#1}}[#2]}
\newcommand\Emph[1]{\textit{#1}\aside{#1}}
\newcommand\EmphE[2]{\textit{#1}\aaside{#1}{#2}}

\title{A Characterization of $(4,2)$-Choosable Graphs}
\author{Daniel W. Cranston\thanks{Department of Mathematics and Applied
Mathematics, Virginia Commonwealth University, Richmond, VA;
\texttt{dcranston@vcu.edu}; 
This research is partially supported by NSA Grant 
H98230-16-0351.}}

\begin{document}
\maketitle
\abstract{A graph $G$ is \emph{$(a,b)$-choosable} if given any list assignment
$L$ with $\card{L(v)}=a$ for each $v\in V(G)$ there exists a function $\vph$ such
that $\vph(v)\in L(v)$ and $\card{\vph(v)}=b$ for all $v\in V(G)$, and whenever
vertices $x$ and $y$ are adjacent $\vph(x)\cap \vph(y)=\emptyset$.  Meng, Puleo,
and Zhu conjectured a characterization of (4,2)-choosable graphs. We prove
their conjecture.
}

\section{Introduction}

\subsection{History}
All graphs we consider are finite and simple (we forbid loops and multiple
edges).
A graph $G$ is \emph{$(a,b)$-choosable}\aside{$(a,b)$-choosable} if given any
list assignment $L$ with
$\card{L(v)}=a$ for each $v\in V(G)$, there exists a function $\vph$ such that
$\vph(v)\in L(v)$ and $\card{\vph(v)}=b$ for all $v\in V(G)$ and $\vph(v)\cap
\vph(w)=\emptyset$ for all $vw\in E(G)$.  
In other words, $\vph$ assigns to each vertex a subset of size $b$ of its $a$
allowable colors, and any adjacent vertices are assigned disjoint subsets.
Such an $L$ is an \emph{$a$-list assignment} (or
\emph{$a$-assignment}\aaside{$a$-assignment}{-.5cm}, for short)
and such a $\vph$ is a \emph{$b$-fold $L$-coloring}\aside{$b$-fold $L$-coloring} of $G$.
As a special case, a graph is
\emph{$(a,b)$-colorable}\aaside{$(a,b)$-colorable}{0cm}
if it has a $b$-fold $L$ coloring, when $L(v)=\{1,\ldots,a\}$ for all $v$.

The concept of $(a,b)$-choosability was introduced in the late 1970s
by \erdos, Rubin, and Taylor~\cite{ERT}.  In the same paper, Rubin
characterized (2,1)-choosable graphs.  To state his result, we
need two definitions.  A \emph{$\theta$-graph},
$\theta_{a,b,c}$\aside{$\theta_{a,b,c}$}, is formed from
vertex disjoint paths with lengths $a$, $b$, $c$ by identifying one endpoint of
each path to form a vertex of degree 3 and also identifying the other endpoint
of each path to form a second vertex of degree 3.  We occasionally write
$\theta_{a,b,c,d}$ for the analogous graph with a fourth path, of length $d$.
The \emph{core}\aside{core} of a
connected graph $G$, denoted $\core(G)$, is its maximum subgraph with minimum
degree at least 2.  Alternately, the core is the result when we
repeatedly delete vertices of degree 1, for as long as possible.  
A graph is (2,1)-choosable if and only if its core is
(this is a special case of Proposition~\ref{prop1}).  Rubin showed that a
connected graph is (2,1)-choosable if and only if its core is $K_1$, $C_{2s}$, or
$\theta_{2,2,2s}$.

\erdos, Rubin, and Taylor concluded their paper with a number of conjectures
and open questions.  One asks whether a graph being $(a,b)$-choosable implies
that it is $(am,bm)$-choosable.  This question remains largely
open\footnote{After this paper was submitted, Dvo\v{r}\'ak, Hu, and Sereni~\cite{DHS}
constructed a (4,1)-choosable graph that is not (8,2)-choosable.}, though some
cases have been resolved, affirmatively.  Tuza and Voigt~\cite{TV} used Rubin's
characterization of (2,1)-choosable graphs to answer this question positively when
$(a,b)=(2,1)$.  In related work, Voigt~\cite{voigt98} showed that when $m$ is an
odd integer, a graph is $(2m,m)$-choosable precisely when it is
$(2,1)$-choosable.  (See Lemma~\ref{GT-lem} for a short proof published
in~\cite{GT}.)
Combining these two results shows that whenever
$a$ and $m$ are integers, with $a$ odd, if a graph is $(2a,a)$-choosable, then
it is $(2am,am)$-choosable.  As far as we know, $(a,b)$-choosable
graphs have been characterized only when $(a,b)=(2m,m)$ and $m$ is odd.

Here we focus mainly on $(4,2)$-choosable graphs.  However, in
Section~\ref{sec:general-m}, we return to the more general case of
$(2m,m)$-choosable graphs, when $m$ is even.
Clearly such a graph must be bipartite, since no odd cycle is
$(2m,m)$-colorable.  Alon, Tuza, and Voigt~\cite{ATV} showed that if $G$ is
$(a,b)$-colorable, then there exists an integer $m$ such that $G$ is
$(am,bm)$-choosable.  Thus, for every bipartite graph $G$, there exists an
$m$ such that $G$ is $(2m,m)$-choosable.  It is easy to construct examples
showing that $m$ must depend on $G$ (see Proposition~\ref{construction-prop}). 
That is, there does not exist a universal constant $m$ such that every
bipartite graph is $(2m,m)$-choosable.

Before moving on, we remark briefly about
$(2m,m)$-paintability.  (Paintability is an online version of list-coloring;
for a definition, see~\cite{MMZ}.)  A connected graph is $(2m,m)$-paintable if
and only if its core is $K_1$, $\theta_{2,2,2}$, or $C_{2s}$.  Zhu~\cite{zhu} proved this
when $m=1$ and it was extended to the general case by Mahoney, Meng, and
Zhu~\cite{MMZ}.  As with choosability, a graph is $(2m,m)$-paintable if and only
if its core is.  That even cycles are $(2m,m)$-paintable follows directly from
the Kernel Lemma (slightly generalized), but handling $\theta_{2,2,2}$ is harder. 
Most of the work in~\cite{MMZ} goes to proving the other direction: for every
$m$ all other graphs are not $(2m,m)$-paintable.

Now we consider $(4,2)$-choosability.
Meng, Puleo, and Zhu~\cite{MPZ} conjectured a characterization of
(4,2)-choosable graphs, and our main result is to confirm their conjecture.
We write $H_1*H_2$ to denote all graphs formed from vertex-disjoint copies of $H_1$
and $H_2$ by adding a path (possibly of length 0) from any vertex in $H_1$ to
any vertex in $H_2$.


\begin{conj}
A connected graph is (4,2)-choosable if and only if its core is one of the
following (where $s$ and $t$ are positive integers): (i) $K_1$,
(ii) $C_{2s}$, (iii) $\theta_{2,2s,2t}$, (iv) $\theta_{1,2s+1,2t+1}$, (v) $K_{2,4}$, 
(vi) a graph formed from $K_{3,3}-e$ by subdividing a single edge incident to a
vertex of degree 2 an even number of times, 
(vii) $C_{2s}*C_{2t}$,
(viii) $\theta_{2,2,2}*C_{2s}$, or
(ix) $(C_4*C_{2s})*C_{2t}$ where the $C_4$ contains two cut-vertices in the
final graph and they are non-adjacent. 
\label{main-conj}
\end{conj}

\begin{main-thm}
Conjecture~\ref{main-conj} is true.
\end{main-thm}

\subsection{Proof Outline}

The proof of the Main Theorem has a number of cases, but the general outline is
easy to follow, so we present it here.  First we need a few more definitions
and a key lemma.

For graphs $G$ and $H$, we say that $G$ \emph{contains a strong minor of
$H$}\aside{strong minor} if
$H$ can be formed from some subgraph of $G$ by repeatedly applying the following
operation: delete a vertex and identify all of its neighbors.  Further, we say
that $G$ \emph{contains a strong subdivision of $H$}\aside{strong subdivision}
if $H$ can be formed from some subgraph of $G$ by repeatedly applying the
following operation: delete a vertex of degree 2 with a neighbor of degree 2
and identify the two neighbors of the deleted vertex.  Clearly, if $G$ contains
a strong subdivision of $H$ then $G$ contains a strong minor of $H$, but not
vice versa.
The following lemma is from~\cite{MPZ}, although a slightly weaker form appeared
in~\cite{voigt98}, and both versions have their roots in~\cite{ERT}, which
contains similar ideas for (2,1)-choosability.  

\begin{StrongMinor}
If $H$ is not $(2m,m)$-choosable, and $G$ contains a strong minor of $H$, then $G$ is
not $(2m,m)$-choosable.
\label{Strong-Minor-lem}
\end{StrongMinor}
\begin{proof}
Suppose that $G'$ is formed from a subgraph of $G$ by deleting a single vertex,
$v$, and identifying its neighbors.  We show that if $G'$ is not
$(2m,m)$-choosable, then neither is $G$.  Let $v'$ be the newly formed vertex
in $G'$.  Let $L'$ be a $2m$-assignment showing that $G'$ is not
$(2m,m)$-choosable.  Form a $2m$-assignment for $G$ as follows.  If $w\in
V(G')$, then $L(w)=L'(w)$.  If $w\in (v\cup\{N(v)\})$, then $L(w)=L'(v')$.  If
$w\notin (V(G')\cup\{v\}\cup N(v))$,
then let $L(w)$ be an arbitrary set of $2m$ colors.  Now suppose that $G$ has an
$m$-fold $L$-coloring, $\vph$.  Note that $\vph(w)=L(v)\setminus \vph(v)$ for
every $w\in N(v)$.  Thus, by deleting $v$ and identifying its neighbors, we get
an $m$-fold $L'$-coloring of $G'$, a contradiction.  Hence, $G$ has no $m$-fold
$L$-coloring.  So $G$ is not $(2m,m)$-choosable.
The lemma follows by induction on the number of deletion/contraction operations
used to form $H$ from a subgraph of $G$.
%
%
\end{proof}

In the rest of the paper all graphs we consider are connected and bipartite.
It suffices to consider the core of $G$, so we assume
$\delta\ge2$\aside{$\delta$} (as usual,
$\delta$ denotes the minimum degree).  
Most of
our work is spent showing that if $G$ does not have one of the forms (i)--(ix)
in Conjecture~\ref{main-conj}, then $G$ is not (4,2)-choosable.  Specifically,
we will find some
subgraph of $G$ that is not (4,2)-choosable.  If $G$ has arbitrarily large
girth, then clearly we must consider some subgraph of arbitrary size to
prove that $G$ is not (4,2)-choosable (since all trees are (4,2)-choosable). 
However, this is not really a problem.  Our idea is to give the same list $L(v)$ to
each vertex $v$ in some connected subgraph $H_1$.  In any valid $L$-coloring of $G$,
every vertex of $H_1$ in one part of the bipartition must get the same colors; likewise
for every vertex of $H_1$ in the other part (and the sets of colors used on the
two parts must partition $L(v)$).  Thus, all of these vertices in one part essentially
function as a single vertex.  We also repeat this list assignment method for other
vertex-disjoint connected subgraphs $H_i$.  (This idea is formalized in the
Strong Minor Lemma, above.) As a result, all of the list
assignments we construct explicitly are for graphs with at most 10 vertices. 

We write $\B(G)$\aside{$\B(G)$} to denote the multiset of blocks of $G$ that contain
a cycle, those that are not $K_2$.  It is straightforward to show that if
$\card{\B(G)}\ge 4$, then $G$ is not (4,2)-choosable.  So most of our work is
for when $\card{\B(G)}\in\{1,2,3\}$.  When $\card{B(G)}=1$ (thus, $G$ is
2-connected, since $\delta\ge 2$), we prove a structural lemma that says that
either $G$ is a graph of the form (i)--(vi) in the conjecture, or else $G$
contains a ``bad'' subgraph.  Next we show that all bad subgraphs are, in
fact, not (4,2)-choosable.  Given a graph $G$ and a 4-assignment $L$, to prove
that $G$ has no 2-fold $L$-coloring $\vph$, we typically assume that $\vph$
exists and reach a contradiction.
This finishes the case $\card{\B(G)}=1$.  It also
helps significantly with the cases $\card{\B(G)}\in\{2,3\}$, since each block
in $\B(G)$ must be of the form (i)--(vi).

Now suppose $\card{\B(G)}=2$, and pick an arbitrary block containing a cycle 
$B_1\in \B(G)$.  We focus on
some cut-vertex $v\in V(B_1)$.  For most $G$, we construct some list assignment
$L$ and show that $G$ has no 2-fold $L$-coloring.  Our idea is to consider each
of the ${4\choose 2}$ possible ways to color $v$ from $L(v)$.  For some of
these ways we show that the coloring cannot be extended to all of $B_1$, and
for the others we show that it cannot be extended to $G\setminus B_1$.  The
only exceptions are when $G$ is of the form (vii) or (viii).  The case
$\card{\B(G)}=3$ is similar, except that now the exceptions are of the form (ix).

\subsection{Preliminaries}

Most of our other definitions are standard, but, for reference, we collect some of
them here.  By \emph{$G$ contains $H$} we mean that $H$ is a subgraph of $G$.
A \emph{$k$-vertex}\aside{$k$-vertex} is a vertex of degree $k$.
An \emph{ear decomposition}\aside{ear decomposition} of a graph $G$ is a
partition of its edges into paths
$P_0,\ldots,P_k$ such that $P_0$ is a single edge, each other $P_i$ is a path
with its endpoints in $\cup_{j=0}^{i-1}P_j$ and its internal vertices (if any)
disjoint from this subgraph.  For a connected graph $G$, a vertex $v\in V(G)$
is a \emph{cut-vertex}\aside{cut-vertex} of $G$ if $G-v$ is disconnected.  If a
connected graph has no cut-vertex, then it is
\emph{2-connected}\aside{2-connected}.  A \emph{block}\aaside{block}{0cm} of a graph
$G$ is a maximal 2-connected subgraph.  
We will use two lemmas of Whitney~\cite{whitney32} about 2-connected graphs. 
We need the second in a slightly more general form than is usually stated, so
we include a short proof.

\begin{lem}
A connected graph $G$ with at least three vertices is 2-connected if and only if
every pair of edges lie on a cycle.
\label{whitney1-lem}
\end{lem}

\begin{lem}
A graph $G$ is 2-connected if and only if $G$ has an ear decomposition.
Further, if $G$ is 2-connected and $H$ is a 2-connected subgraph of $G$, then
$G$ has an ear decomposition that begins with an ear decomposition of $H$.
\label{whitney2-lem}
\end{lem}
\begin{proof}
Let $C$ be an arbitrary cycle in $G$.  Clearly $C$ has an ear decomposition.
Now suppose $H$ is a proper subgraph of $G$, and  $P_0, \ldots, P_k$ is an ear
decomposition of $H$.  Pick $e_1\in E(H)$ and $e_2\in E(G)\setminus E(H)$.
By Lemma~\ref{whitney1-lem}, some cycle $D$ contains $e_1$ and
$e_2$.  Let $P_{k+1}$ be a shortest path along $D$ that contains $e_2$ and has
both endpoints in $H$.  Now $P_0,\ldots,P_k,P_{k+1}$ is an ear
decomposition for a larger subgraph of $G$.  By induction on
$\card{E(G)\setminus E(H)}$, we can extend the ear decomposition to all of $G$.
This proves the first statement.  To prove the second, begin with an
ear decomposition of $H$, and extend it to an ear decomposition of $G$, as in
the proof above.
\end{proof}

Our next proposition is essentially from~\cite{ERT} (and appeared explicitly
in~\cite{TV}).  For completeness, we include the proof. 

\begin{prop}
\label{prop1}
For every graph $G$ and every positive integer $m$, graph $G$ is
$(2m,m)$-choosable if and only if $\core(G)$ is $(2m,m)$-choosable.
\end{prop}

\begin{proof}
We assume $G$ is connected.
One direction is trivial, since $\core(G)\subseteq G$.  For the other, suppose
that $\core(G)$ is $(2m,m)$-choosable.  Let $t=\card{V(G)\setminus V(\core(G))}$.
Let $L$ be a $2m$-assignment for $G$.  We show that $G$ has an $m$-fold
$L$-coloring, by induction on $t$.  The case $t=0$ holds since $\core(G)$
is $(2m,m)$-choosable, by hypothesis.  When $t\ge 1$, $G$ has a 1-vertex, $v$;
let $w$ be the neighbor of $v$.  By induction, $G-v$ has an $m$-fold
$L$-coloring, $\vph$.  Now $\card{L(v)\setminus \vph(w)}\ge 2m-m=m$, so we can
extend $\vph$ to $v$.
%
%
\end{proof}

\section{$\card{\B(G)}=1$}
\label{BG1-sec}
Let
$\Ggood=\{C_{2s},\theta_{2,2s,2t},\theta_{1,2s+1,2t+1},\theta_{2,2,2,2}\}$\aside{$\Ggood$},
where $s$ and $t$ range over all positive integers.
Let $\Gbad=\{\theta_{3,3,3},\theta_{2,2,2,4},K_{3,3},K_{2,5},
Q_3-v\}$\aside{$\Gbad$}, where $Q_3$ denotes the 3-dimensional cube (and $v$ is
an arbitrary vertex, since $Q_3$ is vertex transitive).
Every graph in $\Ggood$ is known to be (4,2)-choosable (we give specific
references at the end of Section~\ref{BG3-sec}).  Later in this section
we show that every graph $G$ is not (4,2)-choosable if either $G\in \Gbad$ or
$G$ contains two cycles that intersect in at most one vertex.  (The graphs
$\theta_{3,3,3}$ and $\theta_{2,2,2,4}$ were shown to not be (4,2)-choosable by
Meng, Puleo, and Zhu, in Section~6 of~\cite{MPZ}; all other graphs in $\Gbad$
and $\Gcycles$ are shown in Figures~\ref{figS}--\ref{figR}.) To conclude this
section, we will determine which strong subdivisions of $K_{3,3}-e$ are
(4,2)-choosable.  Thus, our next lemma plays a central role in our proof of
Conjecture~\ref{main-conj}.

\begin{lem}
\label{lem-struct}
Let $G$ be 2-connected and bipartite.  Either 
(i) $G\in \Ggood$, 
(ii) $G$ contains two cycles that intersect in at most one vertex, 
(iii) $G$ contains a strong subdivision of a graph in $\Gbad$, or 
(iv) $G$ is a strong subdivision of $\Gmixed$.
\end{lem}
\begin{proof}
Suppose the lemma is false, and let $G$ be a counterexample.  Since $G$ is
2-connected, Lemma~\ref{whitney1-lem} implies that $G$ contains some
$\theta$-graph $H$; if possible, pick $H$ to have its three paths of odd
lengths.  For a path $P$, we write \Emph{$\Int(P)$} to denote the set of
interior vertices of $P$, excluding the endpoints.

\textbf{Case 1: Each path of $H$ has odd length.}
Say $H=\theta_{a,b,c}$, with $a\le b\le c$.  Since $G$ contains no strong
subdivision of $\theta_{3,3,3}$, we have $a=1$.  Let $v$ and $w$ denote the
3-vertices in $H$, and let $P_1$ and $P_2$ denote the other two
$v,w$-paths in $H$.  Since $G$ is 2-connected, Lemma~\ref{whitney2-lem} implies
that $G$ has an ear decomposition that
begins with $vw,P_1,P_2$.  Since $G\notin \Ggood$, the ear decomposition
continues with some path $P_3$.  Let $x_3$ and $y_3$ denote the endpoints of
$P_3$.  If $\{x_3,y_3\}=\{v,w\}$, then $G$ contains a strong subdivision of
$\theta_{3,3,3}$, a contradiction.  If $\card{\{x_3,y_3\}\cap \{v,w\}}=1$, then
$G$ contains two cycles intersecting in exactly one vertex, a contradiction.  If
$x_3,y_3\in V(P_1)$, then $G$ contains two vertex disjoint cycles (one in
$P_1\cup P_3$ and one in $P_2+vw$), a contradiction.  So, by symmetry between
$P_1$ and $P_2$, we assume $x_3\in \Int(P_1)$ and $y_3\in \Int(P_2)$.  

If $x_3$
and $y_3$ are in the same part of the bipartition, then $G$ contains a strong
subdivision of $Q_3-v$, a contradiction.  So assume that $x_3$ and $y_3$ are in
opposite parts.  If $P_3$ is the last ear in the decomposition, then $G$ is a
strong subdivision of $K_{3,3}-e$, a contradiction.  So assume there exists another
ear in the decomposition, $P_4$; call its endpoints $x_4$ and $y_4$.  If
$x_4,y_4\in V(P_3)$, then $G$ contains two cycles intersecting in at most one
vertex (one in $P_3\cup P_4$ and one in, say, $P_1+vw$), a contradiction.
Suppose $\card{\{x_4,y_4\}\cap V(P_3)}=1$; say $x_4\in V(P_3)$ and, by symmetry,
$y_4\in V(P_1)$.  Again, $G$ contains two cycles intersecting in at most one
vertex (one in $P_1\cup P_3\cup P_4$ and one in $P_2+vw$), a contradiction.
So $x_4,y_4\notin V(P_3)$.  By the same argument as for $P_3$, we can assume
that $x_4\in \Int(P_2)$ and $y_4\in \Int(P_1)$.  Further, when we walk along
$P_1\cup P_2$, vertices $x_3$ and $y_3$ must alternate with vertices $x_4$ and
$y_4$ (since otherwise $G$ contains vertex disjoint cycles, a contradiction).
So we may assume the vertices appear in the order $x_3,v,x_4, y_3, w, y_4$.

Suppose that $x_3$ and $x_4$ are both in the opposite part of the bipartition
from $v$.  This implies that also $y_3$ and $y_4$ are in the opposite part from
$w$.  Thus, $G$ contains a strong subdivision of $K_{3,3}$, a contradiction (the
branch vertices are $v,w,x_3,x_4,y_3,y_4$).  So at least one of $x_3$ and $y_3$
is in the same part of the bipartition as $v$.  Now $G$ must contain a strong
subdivision of $Q_3-v$; there are two possibilities, both shown in
Figure~\ref{Case1-fig}.
This completes Case 1.

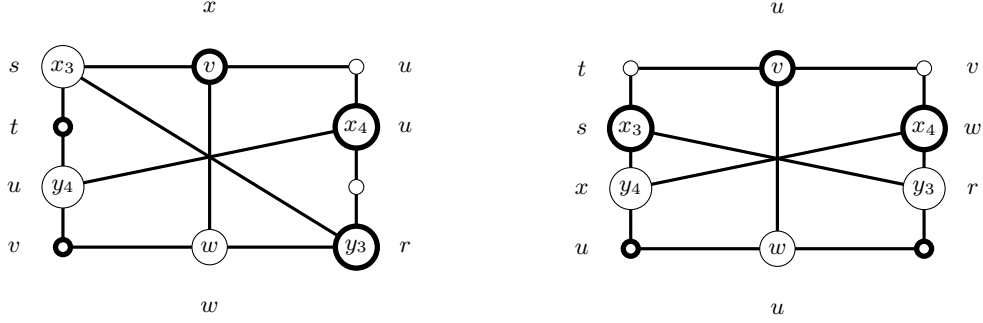
\begin{figure}[t]
\centering
\begin{minipage}[b]{0.45\linewidth}
\centering
\begin{tikzpicture}[xscale = 13, yscale=16]
\Vertex[style = unlabeledStyle2, x = 0.650, y = 0.800, L = \footnotesize {$v$}]{v0}
\Vertex[style = unlabeledStyle, x = 0.500, y = 0.800, L = \footnotesize {$x_3$}]{v1}
\Vertex[style = unlabeledStyle, x = 0.800, y = 0.800, L = \footnotesize {}]{v2}
\Vertex[style = unlabeledStyle2, x = 0.800, y = 0.750, L = \footnotesize {$x_4$}]{v3}
\Vertex[style = unlabeledStyle, x = 0.800, y = 0.700, L = \footnotesize {}]{v4}
\Vertex[style = unlabeledStyle2, x = 0.800, y = 0.650, L = \footnotesize {$y_3$}]{v5}
\Vertex[style = unlabeledStyle, x = 0.650, y = 0.650, L = \footnotesize {$w$}]{v6}
\Vertex[style = unlabeledStyle2, x = 0.500, y = 0.650, L = \footnotesize {}]{v7}
\Vertex[style = unlabeledStyle, x = 0.500, y = 0.700, L = \footnotesize {$y_4$}]{v8}
\Vertex[style = unlabeledStyle2, x = 0.500, y = 0.750, L = \footnotesize {}]{v9}
\Vertex[style = labeledStyle, x = 0.650, y = 0.850, L = \footnotesize {$x$}]{v10}
\Vertex[style = labeledStyle, x = 0.850, y = 0.800, L = \footnotesize {$u$}]{v11}
\Vertex[style = labeledStyle, x = 0.850, y = 0.750, L = \footnotesize {$u$}]{v12}
\Vertex[style = labeledStyle, x = 0.450, y = 0.700, L = \footnotesize {$u$}]{v13}
\Vertex[style = labeledStyle, x = 0.450, y = 0.800, L = \footnotesize {$s$}]{v14}
\Vertex[style = labeledStyle, x = 0.450, y = 0.750, L = \footnotesize {$t$}]{v15}
\Vertex[style = labeledStyle, x = 0.650, y = 0.600, L = \footnotesize {$w$}]{v16}
\Vertex[style = labeledStyle, x = 0.850, y = 0.650, L = \footnotesize {$r$}]{v17}
\Vertex[style = labeledStyle, x = 0.450, y = 0.650, L = \footnotesize {$v$}]{v18}
\Edge[](v1)(v0)
\Edge[](v2)(v0)
\Edge[](v1)(v9)
\Edge[](v8)(v9)
\Edge[](v6)(v7)
\Edge[](v8)(v7)
\Edge[](v4)(v5)
\Edge[](v6)(v5)
\Edge[](v2)(v3)
\Edge[](v4)(v3)
\Edge[](v0)(v6)
\Edge[](v3)(v8)
\Edge[](v1)(v5)
\end{tikzpicture}
\end{minipage}
\begin{minipage}[b]{0.45\linewidth}
\centering
\begin{tikzpicture}[xscale = 13, yscale=16]
\Vertex[style = unlabeledStyle2, x = 0.650, y = 0.800, L = \footnotesize {$v$}]{v0}
\Vertex[style = unlabeledStyle, x = 0.500, y = 0.800, L = \footnotesize {}]{v1}
\Vertex[style = unlabeledStyle, x = 0.800, y = 0.800, L = \footnotesize {}]{v2}
\Vertex[style = unlabeledStyle2, x = 0.800, y = 0.750, L = \footnotesize {$x_4$}]{v3}
\Vertex[style = unlabeledStyle, x = 0.800, y = 0.700, L = \footnotesize {$y_3$}]{v4}
\Vertex[style = unlabeledStyle2, x = 0.800, y = 0.650, L = \footnotesize {}]{v5}
\Vertex[style = unlabeledStyle, x = 0.650, y = 0.650, L = \footnotesize {$w$}]{v6}
\Vertex[style = unlabeledStyle2, x = 0.500, y = 0.650, L = \footnotesize {}]{v7}
\Vertex[style = unlabeledStyle, x = 0.500, y = 0.700, L = \footnotesize {$y_4$}]{v8}
\Vertex[style = unlabeledStyle2, x = 0.500, y = 0.750, L = \footnotesize {$x_3$}]{v9}
\Vertex[style = labeledStyle, x = 0.450, y = 0.700, L = \footnotesize {$x$}]{v10}
\Vertex[style = labeledStyle, x = 0.450, y = 0.650, L = \footnotesize {$u$}]{v11}
\Vertex[style = labeledStyle, x = 0.650, y = 0.850, L = \footnotesize {$u$}]{v12}
\Vertex[style = labeledStyle, x = 0.450, y = 0.750, L = \footnotesize {$s$}]{v13}
\Vertex[style = labeledStyle, x = 0.450, y = 0.800, L = \footnotesize {$t$}]{v14}
\Vertex[style = labeledStyle, x = 0.650, y = 0.600, L = \footnotesize {$u$}]{v15}
\Vertex[style = labeledStyle, x = 0.850, y = 0.700, L = \footnotesize {$r$}]{v16}
\Vertex[style = labeledStyle, x = 0.850, y = 0.800, L = \footnotesize {$v$}]{v17}
\Vertex[style = labeledStyle, x = 0.850, y = 0.750, L = \footnotesize {$w$}]{v18}
\Edge[](v0)(v6)
\Edge[](v1)(v0)
\Edge[](v1)(v9)
\Edge[](v2)(v0)
\Edge[](v2)(v3)
\Edge[](v3)(v8)
\Edge[](v4)(v3)
\Edge[](v4)(v5)
\Edge[](v6)(v5)
\Edge[](v6)(v7)
\Edge[](v8)(v7)
\Edge[](v8)(v9)
\Edge[](v9)(v4)
\end{tikzpicture}
\end{minipage}
\caption{Two strong subdivisions of $Q_3-v$ that arise in Case~1 of the proof of
Lemma~\ref{lem-struct}.  
The boldness of the vertex indicates
its part in the bipartition. The label outside of each vertex indicates its
preimage in $Q_3-v$, as shown in Figure~\ref{figQ}.\label{Case1-fig}}
\end{figure}

\begin{figure}[!b]
\centering
\begin{tikzpicture}[scale = 13]
\Vertex[style = unlabeledStyle2, x = 0.600, y = 0.700, L = \footnotesize {$y_4$}]{v0}
\Vertex[style = unlabeledStyle2, x = 0.600, y = 0.900, L = \footnotesize {$v$}]{v1}
\Vertex[style = unlabeledStyle, x = 0.600, y = 0.800, L = \footnotesize {}]{v2}
\Vertex[style = unlabeledStyle, x = 0.600, y = 0.600, L = \footnotesize {}]{v3}
\Vertex[style = unlabeledStyle2, x = 0.600, y = 0.500, L = \footnotesize {$w$}]{v4}
\Vertex[style = unlabeledStyle, x = 0.500, y = 0.700, L = \footnotesize {}]{v5}
\Vertex[style = unlabeledStyle2, x = 0.400, y = 0.700, L = \footnotesize {$x_4$}]{v6}
\Vertex[style = unlabeledStyle, x = 0.500, y = 0.800, L = \footnotesize {}]{v7}
\Vertex[style = unlabeledStyle, x = 0.800, y = 0.700, L = \footnotesize {}]{v8}
\Vertex[style = unlabeledStyle, x = 0.500, y = 0.600, L = \footnotesize {}]{v9}
\Vertex[style = labeledStyle, x = 0.600, y = 0.950, L = \footnotesize {$u$}]{v10}
\Vertex[style = labeledStyle, x = 0.465, y = 0.835, L = \footnotesize {$t$}]{v11}
\Vertex[style = labeledStyle, x = 0.350, y = 0.700, L = \footnotesize {$s$}]{v12}
\Vertex[style = labeledStyle, x = 0.465, y = 0.565, L = \footnotesize {$r$}]{v13}
\Vertex[style = labeledStyle, x = 0.600, y = 0.450, L = \footnotesize {$w$}]{v14}
\Vertex[style = labeledStyle, x = 0.840, y = 0.700, L = \footnotesize {$v$}]{v15}
\Vertex[style = labeledStyle, x = 0.665, y = 0.700, L = \footnotesize {$x$}]{v16}
\draw[rounded corners=10pt] (.64, .7) -- (.64, .83) -- (.57, .83) -- (.57, .73)
-- (.47, .73) -- (.47, .67) -- (.57, .67) -- (.57, .57) -- (.64, .57) -- (.64, .7);
\Edge[](v0)(v2)
\Edge[](v1)(v2)
\Edge[](v1)(v7)
\Edge[](v6)(v7)
\Edge[](v0)(v5)
\Edge[](v6)(v5)
\Edge[](v4)(v9)
\Edge[](v6)(v9)
\Edge[](v0)(v3)
\Edge[](v4)(v3)
\Edge[](v1)(v8)
\Edge[](v4)(v8)
\end{tikzpicture}
\caption{A strong subdivision of $Q_3-v$ arising in Case~2 of the proof of
Lemma~\ref{lem-struct}.  The boldness of the vertex indicates
its part in the bipartition. The label outside of each vertex (or set of
vertices) indicates its
preimage in $Q_3-v$, as shown in Figure~\ref{figQ}.\label{Case2-fig}}
\end{figure}
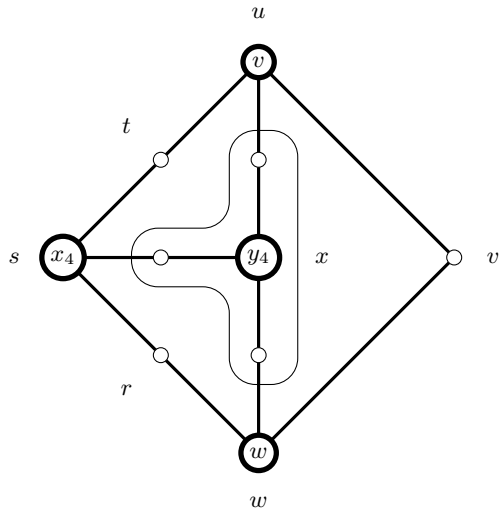

\textbf{Case 2: $G$ has a $\theta$-subgraph with paths of even lengths (but none
with paths of odd lengths).}
Let $P_1$, $P_2$, $P_3$ be the paths of the $\theta$-subgraph, and let $v$ and
$w$ denote their common endpoints.  Suppose that $G$ contains a fourth
vertex disjoint $v,w$-path, $P_4$.  
Consider an ear decomposition of $G$ beginning with $P_1,\ldots,P_4$.
Suppose it continues with some ear $P_5$, and let $x_5,y_5$ denote its
endpoints.  By symmetry, we assume that either (i) $\{x_5,y_5\}=\{v,w\}$, (ii)
$x_5\in \{v,w\}$ and $y_5\in \Int(P_1)$, (iii) $x_5,y_5\in \Int(P_1)$, or (iv)
$x_5\in \Int(P_1)$ and $y_5\in \Int(P_2)$.  In (i), $G$ has a strong subdivision
of $K_{2,5}$, so we are done.  In each of (ii), (iii), and (iv), $G$ has two
cycles that intersect in at most one vertex, so again we are done.  So we
assume no such $P_5$ exists.  Thus, either $G=\theta_{2,2,2,2}$ or $G$ is a
strong subdivision of $\theta_{2,2,2,4}$; in each case we are done.

So assume instead that $G$ has no further $v,w$-path.  Since $G\notin \Ggood$,
the ear decomposition must have a fourth ear, $P_4$; again, denote its endpoints
by $x_4$ and $y_4$.  If $\{x_4,y_4\}\cap\{v,w\}\ne \emptyset$, then $G$ contains
two cycles intersecting in at most one vertex, so we are done.  Similarly, if
$\{x_4,y_4\}\in \Int(P_1)$, then $G$ contains two disjoint cycles; again we are
done.  So, by symmetry, we assume that $x_4\in \Int(P_1)$ and $y_4\in
\Int(P_2)$.  Since, we are not in Case 1, vertices $x_4$ and $y_4$ must be in
the same part of the bipartition.  Similarly, vertices $x_4$ and $y_4$ are in
the same part of the bipartition as vertices $v$ and $w$.  Now $G$ contains a
strong subdivision of $Q_3-v$, a contradiction (see Figure~\ref{figQ}).
\end{proof}

Let $\Gcycles$\aside{$\Gcycles$} denote the set of five graphs shown in
Figures~\ref{figS}, \ref{figT}, \ref{figU}, \ref{figP}, and \ref{figR}.

\begin{lem}
If a graph $G$ is 2-connected, bipartite, and contains two cycles that
intersect in at most one vertex, then $G$ contains a strong minor of some graph
in $\Gcycles$.
\label{lem-cycles-minor}
\end{lem}
\begin{proof}
Let $G$ be 2-connected. Suppose $G$ has two cycles, $C_1$ and
$C_2$, that intersect in a single vertex, $v$.  By Menger's Theorem there exists
a path $P$ from $C_1$ to $C_2$ that has its internal vertices disjoint from
$C_1$ and $C_2$.  Let $w_1$ and $w_2$ denote the endpoints of $P$ on $C_1$ and
$C_2$.  Now $G$ has as a strong minor the graph in Figure~\ref{figS},
\ref{figP}, or~\ref{figR}, depending on whether 0, 1, or 2 of $w_1$ and $w_2$
lie in the same part of the bipartition as $v$.  

Suppose instead that $G$ has vertex disjoint cycles, $C_1$ and $C_2$.  By
Menger's Theorem, $G$ has disjoint paths $P_1$ and $P_2$ from $C_1$ to $C_2$
(with their internal vertices disjoint from $C_1$ and $C_2$).
Let $v_1$ and $w_1$ denote the endpoints on $C_1$ of $P_1$ and
$P_2$, and let $v_2$ and $w_2$ denote the endpoints on $C_2$ of $P_1$ and $P_2$.
If either of $P_1$ and $P_2$ has even length, then we can contract it to reach
the case handled above, since $G$ has as a strong minor two cycles intersecting in
a single vertex.
So we instead assume that both $P_1$ and $P_2$ have odd length.  
Now $G$ has as a strong minor either Figure~\ref{figT} or Figure~\ref{figU},
depending on whether or not $v_1$ and $v_2$ lie in the same part of the
bipartition.
\end{proof}

\begin{figure}[!b]
\centering
\begin{minipage}[b]{0.45\linewidth}
\centering
\begin{tikzpicture}[scale = 12]
\Vertex[style = unlabeledStyle, x = 0.450, y = 0.900, L = \footnotesize {}]{v0}
\Vertex[style = unlabeledStyle, x = 0.600, y = 0.900, L = \footnotesize {}]{v1}
\Vertex[style = unlabeledStyle2, x = 0.600, y = 0.750, L = \footnotesize {}]{v2}
\Vertex[style = unlabeledStyle, x = 0.450, y = 0.750, L = \footnotesize {}]{v3}
\Vertex[style = unlabeledStyle, x = 0.600, y = 0.600, L = \footnotesize {}]{v4}
\Vertex[style = unlabeledStyle, x = 0.450, y = 0.600, L = \footnotesize {}]{v5}
\Vertex[style = unlabeledStyle, x = 0.750, y = 0.600, L = \footnotesize {}]{v6}
\Vertex[style = unlabeledStyle, x = 0.750, y = 0.750, L = \footnotesize {}]{v7}
\Vertex[style = labeledStyle, x = 0.600, y = 0.550, L = \footnotesize {$1236$}]{v8}
\Vertex[style = labeledStyle, x = 0.400, y = 0.550, L = \footnotesize {$1456$}]{v9}
\Vertex[style = labeledStyle, x = 0.380, y = 0.750, L = \footnotesize {$1245$}]{v10}
\Vertex[style = labeledStyle, x = 0.400, y = 0.950, L = \footnotesize {$1235$}]{v11}
\Vertex[style = labeledStyle, x = 0.600, y = 0.950, L = \footnotesize {$1234$}]{v12}
\Vertex[style = labeledStyle, x = 0.650, y = 0.800, L = \footnotesize {$1234$}]{v13}
\Vertex[style = labeledStyle, x = 0.800, y = 0.800, L = \footnotesize {$1345$}]{v14}
\Vertex[style = labeledStyle, x = 0.800, y = 0.550, L = \footnotesize {$1256$}]{v15}
\Edge[](v4)(v2)
\Edge[](v7)(v2)
\Edge[](v4)(v6)
\Edge[](v7)(v6)
\Edge[](v1)(v2)
\Edge[](v3)(v2)
\Edge[](v1)(v0)
\Edge[](v3)(v0)
\Edge[](v5)(v4)
\Edge[](v5)(v3)
\end{tikzpicture}
\caption{\label{figS}}
\end{minipage}
\begin{minipage}[b]{0.45\linewidth}
\centering
\begin{tikzpicture}[scale = 12]
\Vertex[style = unlabeledStyle, x = 0.550, y = 0.800, L = \footnotesize {$w$}]{v0}
\Vertex[style = unlabeledStyle2, x = 0.700, y = 0.700, L = \footnotesize {$x$}]{v1}
\Vertex[style = unlabeledStyle, x = 0.850, y = 0.800, L = \footnotesize {$s$}]{v2}
\Vertex[style = unlabeledStyle, x = 0.700, y = 0.900, L = \footnotesize {$r$}]{v3}
\Vertex[style = unlabeledStyle, x = 0.700, y = 0.550, L = \footnotesize {$u$}]{v4}
\Vertex[style = unlabeledStyle, x = 0.550, y = 0.650, L = \footnotesize {$v$}]{v5}
\Vertex[style = unlabeledStyle, x = 0.850, y = 0.650, L = \footnotesize {$t$}]{v6}
\Vertex[style = labeledStyle, x = 0.700, y = 0.950, L = \footnotesize {$1367$}]{v7}
\Vertex[style = labeledStyle, x = 0.700, y = 0.500, L = \footnotesize {$1456$}]{v8}
\Vertex[style = labeledStyle, x = 0.915, y = 0.650, L = \footnotesize {$1457$}]{v9}
\Vertex[style = labeledStyle, x = 0.915, y = 0.800, L = \footnotesize {$1267$}]{v10}
\Vertex[style = labeledStyle, x = 0.485, y = 0.800, L = \footnotesize {$1234$}]{v11}
\Vertex[style = labeledStyle, x = 0.485, y = 0.650, L = \footnotesize {$1235$}]{v12}
\Vertex[style = labeledStyle, x = 0.700, y = 0.750, L = \footnotesize {$1246$}]{v13}
\Edge[](v0)(v1)
\Edge[](v2)(v1)
\Edge[](v4)(v1)
\Edge[](v2)(v6)
\Edge[](v4)(v6)
\Edge[](v0)(v3)
\Edge[](v2)(v3)
\Edge[](v0)(v5)
\Edge[](v4)(v5)
\end{tikzpicture}
\caption{\label{figQ}}
\end{minipage}
\end{figure}

We typically denote the colors in our 4-assignments by elements of
$\{1,\ldots,7\}$.  For brevity, we usually suppress set notation.  So we write
$1235$ as shorthand for $\{1,2,3,5\}$.

\begin{lem}
Every graph in $\Gbad$ and $\Gcycles$ is not $(4,2)$-choosable.
\label{lem-cycles-bad}
\end{lem}
\begin{proof}
Given a 4-assignment $L$ for a graph $G$, to show that $G$ has no 2-fold
$L$-coloring, we pick some vertex $z$ and show that each of the ${4\choose 2}$
ways to color $z$ cannot be extended to all of $G$.  This is generally 
straightforward, though a few cases involve wrinkles.  
For brevity we omit the details, but they are available in an earlier version of
this paper~\cite{(4:2)arxiv}.
To illustrate the idea, consider the bold vertex, $v$, in Figure~\ref{figS}.  If
$v$ is colored with $12$, then its neighbor to the left, $w$, is colored
$45$, and the vertex below $w$ is colored $16$.  But now the
vertex below $v$ cannot be colored.  Similarly, we can show that none of the
${4\choose 2}$ possible colorings of $v$ extends to the whole graph.  So $G$ has
no $L$-coloring.  For each of Figures~\ref{figS}--\ref{figE}, we use a similar
argument, starting from the bold vertex.  For Figures~\ref{figP}
and~\ref{figR}, the analysis is a bit more subtle, though it follows the same
approach.
\end{proof}

%

\begin{figure}[!h]
\centering
\begin{minipage}[b]{0.45\linewidth}
\centering
\begin{tikzpicture}[scale = 12]
\Vertex[style = unlabeledStyle, x = 0.300, y = 0.900, L = \small {}]{v0}
\Vertex[style = unlabeledStyle, x = 0.450, y = 0.900, L = \small {}]{v1}
\Vertex[style = unlabeledStyle, x = 0.600, y = 0.900, L = \small {}]{v2}
\Vertex[style = unlabeledStyle, x = 0.750, y = 0.900, L = \small {}]{v3}
\Vertex[style = unlabeledStyle, x = 0.750, y = 0.750, L = \small {}]{v4}
\Vertex[style = unlabeledStyle, x = 0.600, y = 0.750, L = \small {}]{v5}
\Vertex[style = unlabeledStyle2, x = 0.450, y = 0.750, L = \small {}]{v6}
\Vertex[style = unlabeledStyle, x = 0.300, y = 0.750, L = \small {}]{v7}
\Vertex[style = labeledStyle, x = 0.300, y = 0.950, L = \footnotesize {$1245$}]{v8}
\Vertex[style = labeledStyle, x = 0.450, y = 0.950, L = \footnotesize {$1356$}]{v9}
\Vertex[style = labeledStyle, x = 0.600, y = 0.950, L = \footnotesize {$1236$}]{v10}
\Vertex[style = labeledStyle, x = 0.750, y = 0.950, L = \footnotesize {$1234$}]{v11}
\Vertex[style = labeledStyle, x = 0.750, y = 0.700, L = \footnotesize {$1456$}]{v12}
\Vertex[style = labeledStyle, x = 0.600, y = 0.700, L = \footnotesize {$1256$}]{v13}
\Vertex[style = labeledStyle, x = 0.450, y = 0.700, L = \footnotesize {$1235$}]{v14}
\Vertex[style = labeledStyle, x = 0.300, y = 0.700, L = \footnotesize {$1234$}]{v15}
\Edge[](v0)(v1)
\Edge[](v2)(v1)
\Edge[](v6)(v1)
\Edge[](v0)(v7)
\Edge[](v6)(v7)
\Edge[](v2)(v5)
\Edge[](v4)(v5)
\Edge[](v6)(v5)
\Edge[](v2)(v3)
\Edge[](v4)(v3)
\end{tikzpicture}
\caption{\label{figT}}
\end{minipage}
\begin{minipage}[b]{0.45\linewidth}
\centering
\begin{tikzpicture}[scale = 12]
\draw[white] (.30,.9)--(.2, .92);
\Vertex[style = unlabeledStyle2, x = 0.450, y = 0.900, L = \small {}]{v0}
\Vertex[style = unlabeledStyle, x = 0.650, y = 0.900, L = \small {}]{v1}
\Vertex[style = unlabeledStyle, x = 0.550, y = 0.750, L = \small {}]{v2}
\Vertex[style = unlabeledStyle, x = 0.650, y = 0.750, L = \small {}]{v3}
\Vertex[style = unlabeledStyle, x = 0.750, y = 0.750, L = \small {}]{v4}
\Vertex[style = unlabeledStyle, x = 0.450, y = 0.750, L = \small {}]{v5}
\Vertex[style = unlabeledStyle, x = 0.350, y = 0.750, L = \small {}]{v6}
\Vertex[style = labeledStyle, x = 0.450, y = 0.950, L = \footnotesize {$1245$}]{v7}
\Vertex[style = labeledStyle, x = 0.650, y = 0.950, L = \footnotesize {$1367$}]{v8}
\Vertex[style = labeledStyle, x = 0.350, y = 0.700, L = \footnotesize {$1234$}]{v9}
\Vertex[style = labeledStyle, x = 0.450, y = 0.700, L = \footnotesize {$1235$}]{v10}
\Vertex[style = labeledStyle, x = 0.550, y = 0.700, L = \footnotesize {$1267$}]{v11}
\Vertex[style = labeledStyle, x = 0.650, y = 0.700, L = \footnotesize {$1456$}]{v12}
\Vertex[style = labeledStyle, x = 0.750, y = 0.700, L = \footnotesize {$1457$}]{v13}
\Edge[](v2)(v1)
\Edge[](v3)(v1)
\Edge[](v4)(v1)
\Edge[](v5)(v1)
\Edge[](v6)(v1)
\Edge[](v2)(v0)
\Edge[](v3)(v0)
\Edge[](v4)(v0)
\Edge[](v5)(v0)
\Edge[](v6)(v0)
\end{tikzpicture}
\caption{\label{figN}}
\end{minipage}
\end{figure}

%

\begin{figure}[!h]
\centering
\begin{minipage}[b]{0.5\linewidth}
\centering
\begin{tikzpicture}[scale = 12]
\Vertex[style = unlabeledStyle2, x = 0.550, y = 0.850, L = \small {}]{v0}
\Vertex[style = unlabeledStyle, x = 0.700, y = 0.850, L = \small {}]{v1}
\Vertex[style = unlabeledStyle, x = 0.750, y = 0.750, L = \small {}]{v2}
\Vertex[style = unlabeledStyle, x = 0.700, y = 0.650, L = \small {}]{v3}
\Vertex[style = unlabeledStyle, x = 0.500, y = 0.750, L = \small {}]{v4}
\Vertex[style = unlabeledStyle, x = 0.550, y = 0.650, L = \small {}]{v5}
\Vertex[style = unlabeledStyle, x = 0.850, y = 0.750, L = \small {}]{v6}
\Vertex[style = unlabeledStyle, x = 0.400, y = 0.750, L = \small {}]{v7}
\Vertex[style = labeledStyle, x = 0.550, y = 0.900, L = \footnotesize {$1345$}]{v8}
\Vertex[style = labeledStyle, x = 0.700, y = 0.900, L = \footnotesize {$1234$}]{v9}
\Vertex[style = labeledStyle, x = 0.685, y = 0.750, L = \footnotesize {$1235$}]{v10}
\Vertex[style = labeledStyle, x = 0.915, y = 0.750, L = \footnotesize {$1245$}]{v11}
\Vertex[style = labeledStyle, x = 0.700, y = 0.600, L = \footnotesize {$1345$}]{v12}
\Vertex[style = labeledStyle, x = 0.550, y = 0.600, L = \footnotesize {$1234$}]{v13}
\Vertex[style = labeledStyle, x = 0.335, y = 0.750, L = \footnotesize {$1245$}]{v14}
\Vertex[style = labeledStyle, x = 0.565, y = 0.750, L = \footnotesize {$1235$}]{v15}
\Edge[](v0)(v4)
\Edge[](v5)(v4)
\Edge[](v0)(v7)
\Edge[](v5)(v7)
\Edge[](v1)(v2)
\Edge[](v3)(v2)
\Edge[](v1)(v6)
\Edge[](v3)(v6)
\Edge[](v1)(v0)
\Edge[](v3)(v5)
\end{tikzpicture}
\caption{\label{figU}}
\end{minipage}
\begin{minipage}[b]{0.4\linewidth}
\centering
\begin{tikzpicture}[scale = 13, rotate=90, xscale=.9, yscale=1.1]
\Vertex[style = unlabeledStyle2, x = 0.450, y = 0.900, L = \small {}]{v0}
\Vertex[style = unlabeledStyle, x = 0.650, y = 0.900, L = \small {}]{v1}
\Vertex[style = unlabeledStyle, x = 0.450, y = 0.800, L = \small {}]{v2}
\Vertex[style = unlabeledStyle, x = 0.650, y = 0.800, L = \small {}]{v3}
\Vertex[style = unlabeledStyle, x = 0.450, y = 0.700, L = \small {}]{v4}
\Vertex[style = unlabeledStyle, x = 0.650, y = 0.700, L = \small {}]{v5}
\Vertex[style = labeledStyle, x = 0.700, y = 0.900, L = \footnotesize {$1345$}]{v6}
\Vertex[style = labeledStyle, x = 0.700, y = 0.800, L = \footnotesize {$1245$}]{v7}
\Vertex[style = labeledStyle, x = 0.700, y = 0.700, L = \footnotesize {$1236$}]{v8}
\Vertex[style = labeledStyle, x = 0.400, y = 0.700, L = \footnotesize {$1235$}]{v9}
\Vertex[style = labeledStyle, x = 0.400, y = 0.800, L = \footnotesize {$1456$}]{v10}
\Vertex[style = labeledStyle, x = 0.400, y = 0.900, L = \footnotesize {$1234$}]{v11}
\Edge[](v1)(v0)
\Edge[](v3)(v0)
\Edge[](v5)(v0)
\Edge[](v1)(v2)
\Edge[](v3)(v2)
\Edge[](v5)(v2)
\Edge[](v1)(v4)
\Edge[](v3)(v4)
\Edge[](v5)(v4)
\end{tikzpicture}
\caption{\label{figE}}
\end{minipage}
\end{figure}

%
%
\begin{figure}[!h]
\centering
\begin{minipage}[b]{0.45\linewidth}
\centering
\begin{tikzpicture}[scale = 12]
\Vertex[style = unlabeledStyle, x = 0.450, y = 0.850, L = \small {}]{v0}
\Vertex[style = unlabeledStyle2, x = 0.600, y = 0.850, L = \small {}]{v1}
\Vertex[style = unlabeledStyle, x = 0.600, y = 0.700, L = \small {}]{v2}
\Vertex[style = unlabeledStyle, x = 0.450, y = 0.700, L = \small {}]{v3}
\Vertex[style = unlabeledStyle, x = 0.950, y = 0.850, L = \small {}]{v4}
\Vertex[style = unlabeledStyle, x = 0.950, y = 0.700, L = \small {}]{v5}
\Vertex[style = unlabeledStyle, x = 0.775, y = 0.775, L = \small {}]{v6}
\Vertex[style = labeledStyle, x = 0.450, y = 0.900, L = \footnotesize {$1235$}]{v7}
\Vertex[style = labeledStyle, x = 0.600, y = 0.900, L = \footnotesize {$1256$}]{v8}
\Vertex[style = labeledStyle, x = 0.950, y = 0.900, L = \footnotesize {$1345$}]{v9}
\Vertex[style = labeledStyle, x = 0.950, y = 0.650, L = \footnotesize {$1234$}]{v10}
\Vertex[style = labeledStyle, x = 0.600, y = 0.650, L = \footnotesize {$1246$}]{v11}
\Vertex[style = labeledStyle, x = 0.450, y = 0.650, L = \footnotesize {$1346$}]{v12}
\Vertex[style = labeledStyle, x = 0.850, y = 0.800, L = \footnotesize {$1236$}]{v13}
\Edge[](v1)(v0)
\Edge[](v3)(v0)
\Edge[](v2)(v1)
\Edge[](v4)(v1)
\Edge[](v6)(v1)
\Edge[](v3)(v2)
\Edge[](v5)(v2)
\Edge[](v5)(v4)
\Edge[](v5)(v6)
\end{tikzpicture}
\caption{\label{figP}}
\end{minipage}
\begin{minipage}[b]{0.45\linewidth}
\centering
\begin{tikzpicture}[scale = 12]
\Vertex[style = unlabeledStyle2, x = 0.550, y = 0.800, L = \small {}]{v0}
\Vertex[style = unlabeledStyle, x = 0.400, y = 0.750, L = \small {}]{v1}
\Vertex[style = unlabeledStyle, x = 0.700, y = 0.750, L = \small {}]{v2}
\Vertex[style = unlabeledStyle, x = 0.700, y = 0.600, L = \small {}]{v3}
\Vertex[style = unlabeledStyle, x = 0.550, y = 0.550, L = \small {}]{v4}
\Vertex[style = unlabeledStyle, x = 0.400, y = 0.600, L = \small {}]{v5}
\Vertex[style = unlabeledStyle, x = 0.300, y = 0.800, L = \small {}]{v6}
\Vertex[style = unlabeledStyle, x = 0.800, y = 0.800, L = \small {}]{v7}
\Vertex[style = labeledStyle, x = 0.300, y = 0.850, L = \footnotesize {$1234$}]{v8}
\Vertex[style = labeledStyle, x = 0.550, y = 0.850, L = \footnotesize {$1345$}]{v9}
\Vertex[style = labeledStyle, x = 0.800, y = 0.850, L = \footnotesize {$1236$}]{v10}
\Vertex[style = labeledStyle, x = 0.400, y = 0.550, L = \footnotesize {$1246$}]{v11}
\Vertex[style = labeledStyle, x = 0.550, y = 0.500, L = \footnotesize {$3456$}]{v12}
\Vertex[style = labeledStyle, x = 0.700, y = 0.550, L = \footnotesize {$2356$}]{v13}
\Vertex[style = labeledStyle, x = 0.650, y = 0.700, L = \footnotesize {$1245$}]{v14}
\Vertex[style = labeledStyle, x = 0.450, y = 0.700, L = \footnotesize {$1256$}]{v15}
\Edge[](v0)(v1)
\Edge[](v5)(v1)
\Edge[](v0)(v6)
\Edge[](v5)(v6)
\Edge[](v4)(v5)
\Edge[](v2)(v3)
\Edge[](v4)(v3)
\Edge[](v0)(v2)
\Edge[](v0)(v7)
\Edge[](v3)(v7)
\end{tikzpicture}
\caption{\label{figR}}
\end{minipage}
\end{figure}



\begin{lem}
A strong subdivision of $\Gmixed$ is (4,2)-choosable only if it can be formed from
$\Gmixed$ by repeatedly subividing a single edge incident to a vertex of degree
2.
\label{lem-mixed}
\end{lem}
\begin{proof}
Note that $\Gmixed$ contains eight edges; four are each incident to one
2-vertex, and the other four are each incident to two 3-vertices.  It is easy to
check that if a strong subdivision of $\Gmixed$ cannot be formed from $\Gmixed$
by repeatedly subdividing a single edge incident to a vertex of degree 2, then
it is a strong subdivision of the graph shown in either Figure~\ref{figY} or
Figure~\ref{figZ}.  So, by the Strong Minor Lemma, it suffices to show that
neither of these graph is (4,2)-choosable.

Let $G$ denote the graph in Figure~\ref{figY}, and $L$ denote its 4-assignment.
If $\vph(x)=12$: $\vph(w)=34$, $\vph(v)=15$, $\vph(u)=26$, $\vph(t)=13$,
$\vph(y)=\none$.  Here, and throughout, we write $\vph(v)=\none$ to denote that
$v$ has at most one color remaining in its list, so cannot be colored.
If $\vph(x)=14$: $\vph(s)=56$, $\vph(r)=13$, $\vph(w)=\none$.
If $\vph(x)=15$: $\vph(y)=23$, $\vph(t)=16$, $\vph(s)=\none$.
If $\vph(x)=24$: $\vph(w)=13$, $\vph(r)=56$, $\vph(s)=\none$.
If $\vph(x)=25$: $\vph(y)=13$, $\vph(t)=26$, $\vph(u)=15$, $\vph(v)=34$,
$\vph(w)=\none$.
If $\vph(x)=45$: $\vph(s)=16$, $\vph(t)=23$, $\vph(y)=\none$.

Now let $G$ denote the graph in Figure~\ref{figZ}, and $L$ denote its 4-assignment.
If $\vph(t)=12$: $\vph(s)=34$, $\vph(r)=15$, $\vph(q)=23$, $\vph(p)=14$,
$\vph(u)=\none$.
If $\vph(t)=14$: $\vph(y)=35$, $\vph(x)=46$, $\vph(w)=23$, $\vph(v)=15$,
$\vph(u)=\none$.
If $\vph(t)=15$: $\vph(u)=24$, $\vph(p)=13$, $\vph(y)=\none$.
If $\vph(t)=24$: $\vph(u)=15$, $\vph(v)=23$, $\vph(w)=46$, $\vph(x)=35$,
$\vph(y)=\none$.
If $\vph(t)=25$: $\vph(u)=14$, $\vph(p)=23$, $\vph(q)=15$, $\vph(r)=34$,
$\vph(s)=\none$.
If $\vph(t)=45$: $\vph(y)=13$, $\vph(p)=24$, $\vph(u)=\none$.
\end{proof}

\begin{figure}
\centering
\begin{minipage}[b]{0.45\linewidth}
\centering
\begin{tikzpicture}[xscale = 10, yscale = 15]
\Vertex[style = unlabeledStyle, x = 0.510, y = 0.750, L = \small {$x$}]{v0}
\Vertex[style = unlabeledStyle, x = 0.655, y = 0.750, L = \small {$y$}]{v1}
\Vertex[style = unlabeledStyle, x = 0.800, y = 0.750, L = \small {$t$}]{v2}
\Vertex[style = unlabeledStyle, x = 0.650, y = 0.600, L = \small {$s$}]{v3}
\Vertex[style = unlabeledStyle, x = 0.650, y = 0.900, L = \small {$w$}]{v4}
\Vertex[style = unlabeledStyle, x = 0.700, y = 0.850, L = \small {$v$}]{v5}
\Vertex[style = unlabeledStyle, x = 0.750, y = 0.800, L = \small {$u$}]{v6}
\Vertex[style = unlabeledStyle, x = 0.320, y = 0.750, L = \small {$r$}]{v7}
\Vertex[style = labeledStyle, x = 0.310, y = 0.720, L = \footnotesize {$1356$}]{v8}
\Vertex[style = labeledStyle, x = 0.720, y = 0.600, L = \footnotesize {$1456$}]{v9}
\Vertex[style = labeledStyle, x = 0.870, y = 0.750, L = \footnotesize {$1236$}]{v10}
\Vertex[style = labeledStyle, x = 0.820, y = 0.800, L = \footnotesize {$1256$}]{v11}
\Vertex[style = labeledStyle, x = 0.770, y = 0.850, L = \footnotesize {$1345$}]{v12}
\Vertex[style = labeledStyle, x = 0.720, y = 0.900, L = \footnotesize {$1234$}]{v13}
\Vertex[style = labeledStyle, x = 0.430, y = 0.750, L = \footnotesize {$1245$}]{v14}
\Vertex[style = labeledStyle, x = 0.650, y = 0.715, L = \footnotesize {$1235$}]{v15}
\Edge[](v7)(v4)
\Edge[](v7)(v3)
\Edge[](v0)(v3)
\Edge[](v0)(v1)
\Edge[](v0)(v4)
\Edge[](v4)(v5)
\Edge[](v6)(v5)
\Edge[](v6)(v2)
\Edge[](v2)(v1)
\Edge[](v2)(v3)
\end{tikzpicture}
\caption{\label{figY}}
\end{minipage}
\begin{minipage}[b]{0.45\linewidth}
\centering
\begin{tikzpicture}[yscale = 11.20, xscale=10]
\Vertex[style = unlabeledStyle, x = 0.550, y = 0.900, L = \footnotesize {$p$}]{v0}
\Vertex[style = unlabeledStyle, x = 0.550, y = 0.700, L = \footnotesize {$x$}]{v1}
\Vertex[style = unlabeledStyle, x = 0.550, y = 0.500, L = \footnotesize {$t$}]{v2}
\Vertex[style = unlabeledStyle, x = 0.650, y = 0.700, L = \footnotesize {$y$}]{v3}
\Vertex[style = unlabeledStyle, x = 0.450, y = 0.700, L = \footnotesize {$w$}]{v4}
\Vertex[style = unlabeledStyle, x = 0.350, y = 0.700, L = \footnotesize {$v$}]{v5}
\Vertex[style = unlabeledStyle, x = 0.250, y = 0.700, L = \footnotesize {$u$}]{v6}
\Vertex[style = unlabeledStyle, x = 0.820, y = 0.700, L = \footnotesize {$r$}]{v7}
\Vertex[style = unlabeledStyle, x = 0.730, y = 0.800, L = \footnotesize {$q$}]{v8}
\Vertex[style = unlabeledStyle, x = 0.730, y = 0.600, L = \footnotesize {$s$}]{v9}
\Vertex[style = labeledStyle, x = 0.375, y = 0.745, L = \footnotesize {$1235$}]{v10}
\Vertex[style = labeledStyle, x = 0.460, y = 0.745, L = \footnotesize {$2346$}]{v11}
\Vertex[style = labeledStyle, x = 0.550, y = 0.745, L = \footnotesize {$3456$}]{v12}
\Vertex[style = labeledStyle, x = 0.720, y = 0.700, L = \footnotesize {$1345$}]{v13}
\Vertex[style = labeledStyle, x = 0.800, y = 0.800, L = \footnotesize {$1235$}]{v14}
\Vertex[style = labeledStyle, x = 0.890, y = 0.700, L = \footnotesize {$1345$}]{v15}
\Vertex[style = labeledStyle, x = 0.800, y = 0.600, L = \footnotesize {$1234$}]{v16}
\Vertex[style = labeledStyle, x = 0.620, y = 0.500, L = \footnotesize {$1245$}]{v17}
\Vertex[style = labeledStyle, x = 0.620, y = 0.900, L = \footnotesize {$1234$}]{v18}
\Vertex[style = labeledStyle, x = 0.235, y = 0.740, L = \footnotesize {$1245$}]{v19}
\Edge[](v2)(v3)
\Edge[](v2)(v9)
\Edge[](v2)(v6)
\Edge[](v0)(v6)
\Edge[](v0)(v3)
\Edge[](v0)(v8)
\Edge[](v7)(v8)
\Edge[](v7)(v9)
\Edge[](v3)(v1)
\Edge[](v4)(v1)
\Edge[](v4)(v5)
\Edge[](v6)(v5)
\end{tikzpicture}
\caption{\label{figZ}}
\end{minipage}
\end{figure}

The proof of our next result is simple (given our work to this point), but the
statement summarizes everything that we have already proved and will need in the
rest of the paper.  So we call it a theorem.
Recall that
$\Ggood=\{C_{2s},\theta_{2,2s,2t},\theta_{1,2s+1,2t+1},\theta_{2,2,2,2}\}$.

\begin{thm}
\label{BG1-lem}
Every 2-connected graph $G$ is not (4,2)-choosable unless either (i) $G\in
\Ggood$ or (ii) $G$ is formed from $\Gmixed$ by
repeatedly subdividing a single edge incident to a 2-vertex.
\end{thm}
\begin{proof}
Suppose $G$ is 2-connected.
If $G$ is (4,2)-choosable, then $G$ is bipartite.
By Lemma~\ref{lem-struct}, either
(i) $G\in \Ggood$, 
(ii) $G$ contains two cycles that intersect in at most one vertex, 
(iii) $G$ contains a strong subdivision of a graph in $\Gbad$, or 
(iv) $G$ is a strong subdivision of $\Gmixed$.  By Lemma~\ref{lem-cycles-minor},
in (ii) $G$ contains a strong minor of a graph in $\Gcycles$.
So, by the Strong Minor Lemma and Lemma~\ref{lem-cycles-bad}, in (ii) and (iii)
$G$ is not (4,2)-choosable.  By Lemma~\ref{lem-mixed}, in (iv) $G$ is not
(4,2)-choosable unless it is formed from $\Gmixed$ by repeatedly subdividing a
single edge incident to a vertex of degree 2.
\end{proof}


\section{$\card{\B(G)}=2$}
\label{BG2-sec}

In this section we consider the case where $G$ has exactly two blocks that
contain cycles, say $B_1$ and $B_2$.  For a cut-vertex $v$, let
\Emph{$V_1,\ldots,V_t$} denote the vertex sets of the components of $G-v$, and
let $G_i=G[V_i\cup\{v\}]$\aside{$G_i$}
for each $i\in [t]$.  (When $|\B(G)|=2$, we will have $t=2$.)
A \EmphE{lollipop}{4mm} is formed from a cycle by identifying the end of a path
(possibly of length 0) with one vertex of the cycle.  

For the sake of illustration, consider standard list-coloring (where each
vertex must be colored with a single color).  Suppose $J$ is formed from two
vertex-disjoint lollipops, say $G_1$ and $G_2$, by identifying their
1-vertices, call them $v_1$ and $v_2$.  To show $G$ is not 2-choosable, it
suffices to construct list assignments $L_1$ and $L_2$ for the two lollipops,
where each $G_i$ is not $L_i$-colorable, $|L_i(v_i)|=1$, and $|L_i(w)|=2$ for
each $w\in V(G_i)\setminus\{v_i\}$.  To form $L$ for $G$, we simply take the
union of the two list assignments, except that if $L_1(v_1)=L_2(v_2)$, then
we permute or rename colors in $L_2$ to avoid this.  Clearly $G$ has no
$L$-coloring, since if we color $v$ with color $\alpha\in L_1(v_1)$, then the
coloring does not extend to $G_1$ and if we color with $\beta\in L_2(v_2)$,
then the coloring does not extend to $G_2$.

We now apply the same idea to show that graphs are not $(4,2)$-choosable.  The
extra complication is that the possible colorings for a cut-vertex $v$ are no
longer independent of each other.  If two possible colorings of $v$ share a
color, then they will also do so after every possible permutation or renaming of
colors.  So if a list assignment forbids some of the ${4\choose 2}$ ways to
color a vertex, then we only care about the relations to each other of the
colorings forbidden for $v$.  For example, two distinct colorings either
intersect in a common color or are disjoint.  A set of three distinct colorings
either (a) intersect in a common color, (b) pairwise intersect, but have no
common intersection, or (c) include two colorings that are complements of each
other, with respect to $L(v)$ (and a third coloring that intersects each).  We
capture this idea with the following definition.

\begin{defn}
\label{def1}
A 4-assignment $L$ for a graph $H$ is \emph{$k$-forcing}\aside{$k$-forcing} for
a vertex $v$ if
every 2-fold $L$-coloring $\vph$ of $H$ assigns $v$ one of at most $k$ subsets
of $L(v)$.  Trivially, every 4-assignment $L$ is 6-forcing for each $v\in V(H)$,
since ${4\choose 2}=6$.  We will be interested in the case when $k\in
\{2,3,4\}$.  For clarity, we write (i) \emph{$2_{in}$-forcing}, (ii)
\emph{$2_{comp}$-forcing}, (iii) \emph{$3_{in}$-forcing}, (iv)
\emph{$3_{out}$-forcing}, and (v) \emph{$4_{out}$-forcing} (again, for a vertex
$v$).  This denotes that (i) the two options for $\vph(v)$ have a common color,
(ii) the two options for $\vph(v)$ are complements of each other (with respect
to $L(v)$), (iii) the three options for $\vph(v)$ have a common color, (iv) the
three options for $\vph(v)$ exclude a common color, and (v) the two excluded
options for $\vph(v)$ have a common color.
\end{defn}

We illustrate the point of this definition with an example.
Suppose we have a graph $G$ with $|\B(G)|=2$, a cut-vertex $v$, and the resulting
subgraphs $G_1$ and $G_2$.  To show that $G$ is not (4,2)-choosable, it suffices
to show that $G_1$ has a $3_{in}$-forcing 4-assignment $L_1$ and $G_2$ has a
$3_{out}$-forcing 4-assignment $L_2$.  To see why, note that by renaming and
permuting color classes, we can assume that $L_1(v)=L_2(v)=1234$ and $L_1$
forces a coloring of $v$ in $\{12, 13, 14\}$, but $L_2$ forces one in $\{23,
24, 34\}$.  Thus, no coloring of $v$ extends to both $G_1$ and $G_2$.  So $G$
has no $L$-coloring.  A similar idea works if $L_1$ is $4_{out}$-forcing and
$L_2$ is $2_{in}$-forcing.  (It is easy to construct a $4_{out}$-forcing
4-assignment for every lollipop; see Lemma~\ref{C4-lem}.  If there exists $v$
such that $G_1$, $G_2$, and $G_3$ each contain a lollipop, then $G$ is not
(4,2)-choosable: we use one lollipop to forbid 12 and 13 on $v$, another to
forbid 23 and 24, and the third to forbid 14 and 34.  Thus, the more
challenging case when $|\B(G)|\ge 3$ is when the block tree is a path.)

The following proposition allows us to extend a $k$-forcing assingment along a
path.  We will use it to show that if $e$ is a cut-edge of $G$, then $G/e$ is
(4,2)-choosable if and only if $G$ is (4,2)-choosable.

\begin{prop}
\label{obs1}
Suppose that there exists $x\in V(H)$ with $d(x)=1$ and $w$ is the neighbor of
$x$.  If $H-x$ has a 4-assignment $L'$ that is $4_{out}$-forcing
(resp.~$2_{in}$-forcing, $2_{comp}$-forcing, $3_{in}$-forcing, and
$3_{out}$-forcing) for $w$, then $H$ has a 4-assignment $L$ that is
$4_{out}$-forcing (resp.~$2_{in}$-forcing, $2_{comp}$-forcing,
$3_{out}$-forcing, and $3_{in}$-forcing) for $x$.  
\end{prop}
\begin{proof}
Let $L$ be given by $L(v)=L'(v)$ for all $v\in V(H)-x$ and $L(x)=L'(w)$.
\end{proof}

Note that in the first three cases $L$ is the same type of forcing assignment
for $x$ as $L'$ is for $w$.  However, the types swap for $3_{out}$-forcing and
$3_{in}$-forcing.

\begin{lem}
\label{C4-lem}
The 4-assignment 1234, 1234, 1235, 2345 for a 4-cycle is $4_{out}$-forcing for
the first two vertices.  Further, if $G$ consists of a 4-cycle with a path
pendant at one vertex, then there exists a 4-assignment that is
$4_{out}$-forcing for each vertex of the path and also for the degree 3 vertex
and one of its neighbors on the cycle.
\end{lem}
\begin{proof}
Denote the vertices of the cycle by $v_1,v_2,v_3,v_4$, in order.  Let $L$ denote
the given list assignment.  It is easy to
check that if $\vph(v_1)\in\{24,34\}$, then we cannot complete the coloring.
Since $L(v_1)=L(v_2)$, this implies that if $\vph(v_2)\in \{13,12\}$, then we
cannot complete the coloring.  This proves the first statement.  For the second
statement, add a path pendant at $v_1$.  Now extend the 4-assignment $L$ by
letting $L(v)=1234$ for each vertex on the path.  Now the second statement
follows by induction on the path length, using Proposition~\ref{obs1}.
\end{proof}

It is enlightening to know that if $G$ consists of an even cycle with a path
pendant at one cycle vertex and $v\in V(G)$, then there is no 4-assignment that
is 3-forcing for $v$.  (This is an easy consequence of the fact that if
$\B(G)=\{C_{2s},C_{2t}\}$, then $G$ is (4,2)-choosable, which was proved by
Meng, Puleo, and Zhu.)  However, we will not need this result until later, so we
prove it as Corollary~\ref{barbell-cor}.

The following lemma is the main result of this section.  

\begin{lem}
\label{BG2-lem}
Let $G$ be a graph with $\delta\ge 2$ and $\card{\B(G)}=2$.  If $\B(G)\ne
\{C_{2s},C_{2t}\}$ and $\B(G)\ne \{C_{2s}, \theta_{2,2,2}\}$, then $G$ is not
$(4,2)$-choosable.
\end{lem}
\begin{proof}
We begin by proving a series of claims.  The point of each is to construct a
$k$-forcing assingment for a 2-connected graph (that is either in $\Ggood$ or is
formed from $K_{3,3}$ by repeatedly subdividing a single edge incident to a
2-vertex).  For a cut-vertex $v$ and subgraphs $G_1$ and $G_2$, our goal is to
find list assingments $L_1$ for $G_1$ and $L_2$ for $G_2$ such that $L_1$ is
$a$-forcing and $L_2$ is $b$-forcing and $a+b\le 6$.  Using the approach
outlined after Definition~\ref{def1}, this allows us to show $G$ is not
(4,2)-choosable.  This approach succeeds unless $\B(G)$ is one of the
two exceptions in the statement of theorem.

\begin{clm}
\label{clm1}
If $H=\theta_{2,2,2}$ and $v,w\in V(H)$ with $v$ and $w$ non-adjacent,  then
$H$ has a 4-assignment that is $3_{in}$-forcing for $v$ and also has a
4-assignment that is $3_{out}$-forcing for $v$.  In one of these
assignments, at least one of the three allowable colorings of $v$ forces a
unique coloring of $w$.
\end{clm}
We will not need the second statement in the proof of the current lemma, but
will use it later on, and it is convenient to prove now.
The desired assignments are shown in Figures~\ref{figWW} and~\ref{figXX}.  By
symmetry, we can assume that
$v\in \{v_1,v_2\}$ (as shown in the figures), and thus $w\in\{w_1,w_2\}$ (as
in Figure~\ref{figXX}).  First consider the labeling $L$ in Figure~\ref{figWW}.
If $\vph(v_1)\in\{12,13,23\}$, then colors 4, 5, 6 are used on its neighbors,
so the final vertex has no coloring.  Thus $L$ is $3_{in}$-forcing for $v_1$.  Since
$L(v_1)=L(v_2)$, also $L$ is $3_{out}$-forcing for $v_2$.

Now consider the assignment $L$ in Figure~\ref{figXX}.  It is easy to check that 
this $L$ forces $\vph(v_1)\in \{13,14,34\}$.  This, in turn, forces
$\vph(v_2)\in\{24,23, 12\}$, since $L(v_1)=L(v_2)$.  So $L$ is
$3_{out}$-forcing for $v_1$ and $3_{in}$-forcing for $v_2$.  Further, 
$\vph(v_1)=13$ if and only if $\vph(v_2)=24$, and if $\vph(v_1)=13$, then
$\vph(w_2)=25$, and $\vph(w_1)=13$.
This finishes the proof of the first statement, and also proves the second statement.  
\sqed

\begin{figure}[!h]
\tikzstyle{unlabeledStyle}=[shape = circle, minimum size = 5pt, inner sep = 1.6pt, outer sep = 0pt, draw]
\centering
\begin{minipage}[b]{0.45\linewidth}
\centering
\begin{tikzpicture}[scale = 12]
\Vertex[style = unlabeledStyle, x = 0.600, y = 0.900, L = \footnotesize {$v_2$}]{v0}
\Vertex[style = unlabeledStyle, x = 0.600, y = 0.750, L = \footnotesize {}]{v1}
\Vertex[style = unlabeledStyle, x = 0.450, y = 0.750, L = \footnotesize {}]{v2}
\Vertex[style = unlabeledStyle, x = 0.750, y = 0.750, L = \footnotesize {$v_1$}]{v3}
\Vertex[style = unlabeledStyle, x = 0.600, y = 0.600, L = \footnotesize {}]{v4}
\Vertex[style = labeledStyle, x = 0.600, y = 0.950, L = \footnotesize {$1234$}]{v5}
\Vertex[style = labeledStyle, x = 0.600, y = 0.800, L = \footnotesize {$1236$}]{v6}
\Vertex[style = labeledStyle, x = 0.600, y = 0.550, L = \footnotesize {$1235$}]{v7}
\Vertex[style = labeledStyle, x = 0.820, y = 0.750, L = \footnotesize {$1234$}]{v8}
\Vertex[style = labeledStyle, x = 0.390, y = 0.750, L = \footnotesize {$1456$}]{v9}
\Edge[](v0)(v2)
\Edge[](v0)(v3)
\Edge[](v1)(v2)
\Edge[](v1)(v3)
\Edge[](v4)(v2)
\Edge[](v4)(v3)
\end{tikzpicture}
\caption{A 4-assignment that is $3_{in}$-forcing for $v_1$.\label{figWW}}
\end{minipage}
\begin{minipage}[b]{0.45\linewidth}
\centering
\begin{tikzpicture}[scale = 12]
\Vertex[style = unlabeledStyle, x = 0.600, y = 0.900, L = \footnotesize {$v_2$}]{v0}
\Vertex[style = unlabeledStyle, x = 0.600, y = 0.750, L = \footnotesize {}]{v1}
\Vertex[style = unlabeledStyle, x = 0.450, y = 0.750, L = \footnotesize {$w_1$}]{v2}
\Vertex[style = unlabeledStyle, x = 0.750, y = 0.750, L = \footnotesize {$v_1$}]{v3}
\Vertex[style = unlabeledStyle, x = 0.600, y = 0.600, L = \footnotesize {$w_2$}]{v4}
\Vertex[style = labeledStyle, x = 0.600, y = 0.950, L = \footnotesize {$1234$}]{v5}
\Vertex[style = labeledStyle, x = 0.600, y = 0.800, L = \footnotesize {$1245$}]{v6}
\Vertex[style = labeledStyle, x = 0.600, y = 0.550, L = \footnotesize {$1235$}]{v7}
\Vertex[style = labeledStyle, x = 0.820, y = 0.750, L = \footnotesize {$1234$}]{v8}
\Vertex[style = labeledStyle, x = 0.380, y = 0.750, L = \footnotesize {$1345$}]{v9}
\Edge[](v0)(v3)
\Edge[](v1)(v3)
\Edge[](v4)(v3)
\Edge[](v0)(v2)
\Edge[](v1)(v2)
\Edge[](v4)(v2)
\end{tikzpicture}
\caption{A 4-assignment that is $3_{out}$-forcing for $v_1$.\label{figXX}}
\end{minipage}
\end{figure}

\begin{clm}
\label{clm2}
If $H=K_{2,4}$ and $v\in V(H)$, then $H$ has a 4-assignment $L$ that is
$2_{in}$-forcing for $v$.
\end{clm}

Consider the 4-assignment $L$ shown in Figure~\ref{figUU}. 
By symmetry (between $r$ and $s$ and also between $t$, $u$, $v$, and $w$), it
suffices to show that $L$ is $2_{in}$-forcing for $r$, since $L(w)=L(r)$.  We
show that $L$ forces $\vph(r)\in \{14,24\}$.  If $\vph(r)\subseteq 123$, then
$u$, $v$, and $w$ use 6, 5, and 4, so we cannot extend the coloring to $s$.  If
$\vph(r)=34$, then $w$ and $t$ use 1, 5, and 6, so we again cannot color $s$. 
Thus, $\vph(r)\in\{14,24\}$, so $L$ is $2_{in}$-forcing for $r$, and also $w$,
which proves Claim~\ref{clm2}. \sqed

\begin{figure}[!h]
\centering
\begin{minipage}[b]{0.45\linewidth}
\centering
\begin{tikzpicture}[scale = 12]
\Vertex[style = unlabeledStyle, x = 0.550, y = 0.850, L = \footnotesize {$r$}]{v0}
\Vertex[style = unlabeledStyle, x = 0.600, y = 0.700, L = \footnotesize {$v$}]{v1}
\Vertex[style = unlabeledStyle, x = 0.500, y = 0.700, L = \footnotesize {$w$}]{v2}
\Vertex[style = unlabeledStyle, x = 0.700, y = 0.700, L = \footnotesize {$u$}]{v3}
\Vertex[style = unlabeledStyle, x = 0.800, y = 0.700, L = \footnotesize {$t$}]{v4}
\Vertex[style = unlabeledStyle, x = 0.750, y = 0.850, L = \footnotesize {$s$}]{v5}
\Vertex[style = labeledStyle, x = 0.550, y = 0.900, L = \footnotesize {$1234$}]{v6}
\Vertex[style = labeledStyle, x = 0.750, y = 0.900, L = \footnotesize {$1456$}]{v7}
\Vertex[style = labeledStyle, x = 0.500, y = 0.650, L = \footnotesize {$1234$}]{v8}
\Vertex[style = labeledStyle, x = 0.700, y = 0.650, L = \footnotesize {$1236$}]{v9}
\Vertex[style = labeledStyle, x = 0.600, y = 0.650, L = \footnotesize {$1235$}]{v10}
\Vertex[style = labeledStyle, x = 0.800, y = 0.650, L = \footnotesize {$3456$}]{v11}
\Edge[](v1)(v5)
\Edge[](v2)(v5)
\Edge[](v3)(v5)
\Edge[](v4)(v5)
\Edge[](v1)(v0)
\Edge[](v2)(v0)
\Edge[](v3)(v0)
\Edge[](v4)(v0)
\end{tikzpicture}
\caption{A 4-assignment that is $2_{in}$-forcing for $r$.\label{figUU}}
\end{minipage}
\begin{minipage}[b]{0.45\linewidth}
\centering
\begin{tikzpicture}[scale = 10]
\Vertex[style = unlabeledStyle, x = 0.450, y = 0.850, L = \footnotesize {$s$}]{v0}
\Vertex[style = unlabeledStyle, x = 0.600, y = 0.850, L = \footnotesize {$t$}]{v1}
\Vertex[style = unlabeledStyle, x = 0.750, y = 0.850, L = \footnotesize {$u$}]{v2}
\Vertex[style = unlabeledStyle, x = 0.750, y = 0.700, L = \footnotesize {$v$}]{v3}
\Vertex[style = unlabeledStyle, x = 0.600, y = 0.700, L = \footnotesize {$w$}]{v4}
\Vertex[style = unlabeledStyle, x = 0.450, y = 0.700, L = \footnotesize {$x$}]{v5}
\Vertex[style = labeledStyle, x = 0.450, y = 0.650, L = \footnotesize {$1345$}]{v6}
\Vertex[style = labeledStyle, x = 0.600, y = 0.650, L = \footnotesize {$1235$}]{v7}
\Vertex[style = labeledStyle, x = 0.750, y = 0.650, L = \footnotesize {$1234$}]{v8}
\Vertex[style = labeledStyle, x = 0.750, y = 0.900, L = \footnotesize {$1346$}]{v9}
\Vertex[style = labeledStyle, x = 0.600, y = 0.900, L = \footnotesize {$1256$}]{v10}
\Vertex[style = labeledStyle, x = 0.450, y = 0.900, L = \footnotesize {$1246$}]{v11}
\Edge[](v0)(v1)
\Edge[](v2)(v1)
\Edge[](v4)(v1)
\Edge[](v0)(v5)
\Edge[](v4)(v5)
\Edge[](v2)(v3)
\Edge[](v4)(v3)
\end{tikzpicture}
\caption{A 4-assignment that is $2_{in}$-forcing for $v$.\label{figVV}}
\end{minipage}
\end{figure}

\begin{clm}
\label{clm3}
If $H=\theta_{1,3,3}$ and $v\in V(H)$, then $H$ has a 4-assignment $L$
that is $2_{in}$-forcing for $v$.
\end{clm}

By symmetry, there are only two possibilities for $v$.  First suppose $d(v)=3$.
Let $w$ denote the other vertex such that $d(w)=3$.  By Lemma~\ref{C4-lem}, give a
4-assignment $L_1$ to one 4-cycle in $H$ such that $L_1(v)=L_1(w)=1234$ and $L_1$
forces $\vph(v)\notin\{12,13\}$.  Similarly, give a 4-assignment $L_2$ to the
other 4-cycle such that $L_2(v)=L_2(w)=1234$ and $L_2$ forces
$\vph(v)\notin\{14,24\}$.  Now $L_1\cup L_2$ forces $\vph(v)\in \{23,34\}$, so
is $2_{in}$-forcing for $v$.

Assume instead that $d(v)=2$.  Let $L$ be the 4-assignment in
Figure~\ref{figVV}.  
It is straightforward to check that $L$ forces
$\vph(v)\in\{14,24\}$, as follows.  If $\vph(v)=12$: $\vph(w)=35$, $\vph(x)=14$,
$\vph(s)=26$, $\vph(t)=\none$.  If $\vph(v)=13$: $\vph(w)=25$, $\vph(t)=16$,
$\vph(u)=\none$.  If $\vph(v)=23$: $\vph(w)=15$, $\vph(t)=26$, $\vph(s)=14$,
$\vph(x)=\none$.  If $\vph(v)=34$: $\vph(u)=16$, $\vph(t)=25$, $\vph(w)=\none$.
This proves Claim~\ref{clm3}. \sqed

\begin{clm}
\label{clm4}
If $H=\theta_{2,2,4}$ and $v\in V(H)$, then $H$ has a 4-assignment $L$
that is $2_{in}$-forcing for $v$.
\end{clm}

Consider $\theta_{2,2,4}$, shown in Figures~\ref{figAA} and~\ref{figBB}.
For any given vertex in $V(\theta_{2,2,4})$, we must construct a 4-assignment
$L$ such that $L$ is $2_{in}$-forcing for that vertex.  By symmetry, we assume
this vertex is $r$, $u$, $v$, or $w$.  Consider the list assignment $L$ in
Figure~\ref{figAA}.

If $\vph(w)=12$: $\vph(v)=34$, $\vph(u)=15$, $\vph(t)=23$, $\vph(s)=14$,
$\vph(r)=\none$.  If $\vph(w)=15$: $\vph(r)=24$, $\vph(s)=13$, $\vph(x)=\none$.
If $\vph(w)=25$: $\vph(r)=14$, $\vph(s)=23$, $\vph(t)=15$, $\vph(u)=34$,
$\vph(v)=\none$.  If $\vph(w)=45$: $\vph(x)=13$, $\vph(s)=24$, $\vph(r)=\none$.
Thus, $\vph(w)\in\{14,24\}$.  Since $\vph(r)=\vph(w)$, we get
$\vph(r)\in\{25,15\}$.  Also, $\vph(v)\in\{23,13\}$.  So $L$ is
$2_{in}$-forcing for $r$, $v$, and $w$.

Now consider the list assignment $L$ in Figure~\ref{figBB}.  We show that it is
$2_{in}$-forcing for $u$ (in fact we show that it is 1-forcing for $u$, but we
will not need this).
We begin by showing $\vph(u)\notin\{35,45,46\}$, which is straightforward.
If $\vph(u)=35$: $\vph(t)=14$, $\vph(s)=23$, $\vph(x)=15$, $\vph(w)=26$,
$\vph(v)=\none$.  If $\vph(u)=45$: $\vph(v)=26$, $\vph(w)=15$, $\vph(x)=23$,
$\vph(s)=14$, $\vph(t)=\none$.  If $\vph(u)=46$: $\vph(v)=25$, $\vph(w)=16$,
$\vph(r)=24$, $\vph(s)=13$, $\vph(t)=\none$.
Now we consider $\vph(u)\in \{34,56\}$.
If $\vph(u)=34$, then $\vph(t)=15$, so $1\notin\vph(s)$.  Also, $4\notin
\vph(v)$, so $1\in \vph(w)$, since otherwise $\vph(v)\cup \vph(w)\subseteq 256$,
a contradiction.  Thus, $1\notin \vph(r)$ and $1\notin \vph(x)$.
If $2\in \vph(s)$, then $\vph(r)=46$ and $\vph(x)=35$, so we cannot color $s$.
Thus, $2\notin \vph(s)$, so $\vph(s)=34$.  Now $\vph(r)=26$ and $\vph(x)=25$, so
we cannot color $w$, a contradiction.
If $\vph(u)=56$, then $\vph(v)=24$, so $2\notin\vph(w)$.  Also, $5\notin
\vph(t)$, so $2\in \vph(s)$, so $2\notin \vph(r)$ and $2\notin \vph(x)$.  
If $1\in \vph(w)$, then $\vph(r)=46$ and $\vph(x)=35$, so we cannot color $w$. 
Thus, $1\notin \vph(w)$, so $\vph(w)=56$.  Now $\vph(x)=13$ and $\vph(r)=14$,
so we cannot color $s$.  This finishes the proof of Claim~\ref{clm4}. \sqed
\bigskip

\begin{figure}[!h]
\centering
\begin{minipage}[b]{0.45\linewidth}
\centering
\begin{tikzpicture}[scale = 12]
\Vertex[style = unlabeledStyle, x = 0.450, y = 0.750, L = \footnotesize {$r$}]{v0}
\Vertex[style = unlabeledStyle, x = 0.600, y = 0.750, L = \footnotesize {$x$}]{v1}
\Vertex[style = unlabeledStyle, x = 0.600, y = 0.900, L = \footnotesize {$s$}]{v2}
\Vertex[style = unlabeledStyle, x = 0.750, y = 0.850, L = \footnotesize {$t$}]{v3}
\Vertex[style = unlabeledStyle, x = 0.750, y = 0.750, L = \footnotesize {$u$}]{v4}
\Vertex[style = unlabeledStyle, x = 0.750, y = 0.650, L = \footnotesize {$v$}]{v5}
\Vertex[style = unlabeledStyle, x = 0.600, y = 0.600, L = \footnotesize {$w$}]{v6}
\Vertex[style = labeledStyle, x = 0.600, y = 0.950, L = \footnotesize {$1234$}]{v7}
\Vertex[style = labeledStyle, x = 0.810, y = 0.850, L = \footnotesize {$1235$}]{v8}
\Vertex[style = labeledStyle, x = 0.810, y = 0.750, L = \footnotesize {$1345$}]{v9}
\Vertex[style = labeledStyle, x = 0.810, y = 0.650, L = \footnotesize {$1234$}]{v10}
\Vertex[style = labeledStyle, x = 0.600, y = 0.550, L = \footnotesize {$1245$}]{v11}
\Vertex[style = labeledStyle, x = 0.385, y = 0.750, L = \footnotesize {$1245$}]{v12}
\Vertex[style = labeledStyle, x = 0.660, y = 0.750, L = \footnotesize {$1345$}]{v13}
\Edge[](v2)(v1)
\Edge[](v6)(v1)
\Edge[](v2)(v0)
\Edge[](v6)(v0)
\Edge[](v2)(v3)
\Edge[](v4)(v3)
\Edge[](v4)(v5)
\Edge[](v6)(v5)
\end{tikzpicture}
\caption{A 4-assignment that is $2_{in}$-forcing for $r$, $v$, and $w$.\label{figAA}}
\end{minipage}
\begin{minipage}[b]{0.45\linewidth}
\centering
\begin{tikzpicture}[scale = 12]
\Vertex[style = unlabeledStyle, x = 0.450, y = 0.750, L = \footnotesize {$r$}]{v0}
\Vertex[style = unlabeledStyle, x = 0.600, y = 0.750, L = \footnotesize {$x$}]{v1}
\Vertex[style = unlabeledStyle, x = 0.600, y = 0.900, L = \footnotesize {$s$}]{v2}
\Vertex[style = unlabeledStyle, x = 0.750, y = 0.850, L = \footnotesize {$t$}]{v3}
\Vertex[style = unlabeledStyle, x = 0.750, y = 0.750, L = \footnotesize {$u$}]{v4}
\Vertex[style = unlabeledStyle, x = 0.750, y = 0.650, L = \footnotesize {$v$}]{v5}
\Vertex[style = unlabeledStyle, x = 0.600, y = 0.600, L = \footnotesize {$w$}]{v6}
\Vertex[style = labeledStyle, x = 0.600, y = 0.950, L = \footnotesize {$1234$}]{v7}
\Vertex[style = labeledStyle, x = 0.810, y = 0.850, L = \footnotesize {$1345$}]{v8}
\Vertex[style = labeledStyle, x = 0.810, y = 0.750, L = \footnotesize {$3456$}]{v9}
\Vertex[style = labeledStyle, x = 0.810, y = 0.650, L = \footnotesize {$2456$}]{v10}
\Vertex[style = labeledStyle, x = 0.600, y = 0.550, L = \footnotesize {$1256$}]{v11}
\Vertex[style = labeledStyle, x = 0.385, y = 0.750, L = \footnotesize {$1246$}]{v12}
\Vertex[style = labeledStyle, x = 0.660, y = 0.750, L = \footnotesize {$1235$}]{v13}
\Edge[](v2)(v0)
\Edge[](v2)(v1)
\Edge[](v2)(v3)
\Edge[](v4)(v3)
\Edge[](v4)(v5)
\Edge[](v6)(v0)
\Edge[](v6)(v1)
\Edge[](v6)(v5)
\end{tikzpicture}
\caption{A 4-assignment that is $2_{in}$-forcing for $u$.\label{figBB}}
\end{minipage}
\end{figure}

Suppose $G$ is $(4,2)$-choosable.  By Theorem~\ref{BG1-lem}, we assume each
block in $\B(G)$ is either (i) $C_{2s}$, (ii) $\theta_{2,2s,2t}$, (iii)
$\theta_{1,2s+1,2t+1}$, (iv) $\theta_{2,2,2,2}$ (that is, $K_{2,4}$) or (v) a
strong subdivision of $\Gmixed$.  Since every instance of (v) contains as a
subgraph an instance of (iii), we need not consider (v).

Suppose that one block, say $B_1$, of $G$ is $K_{2,4}$.   By the Strong Minor
Lemma, it suffices to consider the case that the other block, $B_2$ is an even
cycle.  Let $B_2'$ denote the subgraph of $G$ consisting of $B_2$ and the path
(possibly of length 0) from $B_2$ to $B_1$.  Let $\{v\}=V(B_1)\cap V(B_2')$. 
By Lemma~\ref{C4-lem}, subgraph $B_2'$ has a 4-assignment $L_1$ that is
$4_{out}$-forcing on $v$.  By symmetry, assume that $L_1(v)=1234$ and $L_1$
forces $\vph(v)\notin\{12,13\}$.  By
Claim~\ref{clm2}, block $B_1$ has a 4-assignment $L_2$ that is $2_{in}$-forcing on
$v$.  By permuting colors in $L_2$, we can assume that $L_2(v)=1234$ and $L_2$
forces $\vph(v)\in \{12,13\}$.  Hence, $L_1\cup L_2$ is a 4-assignment $L$ for
$G$ such that $G$ has no 2-fold $L$-coloring.  Thus, $G$ is not (4,2)-choosable,
a contradiction.  So no block of $G$ is $K_{2,4}$.

Suppose that one block, say $B_1$, of $G$ is $\theta_{1,2s+1,2t+1}$.  It
suffices to consider the case when $B_1=\theta_{1,3,3}$.  Let $v$ be the
cut-vertex of $G$ in $B_1$.  By Claim~\ref{clm3}, $B_1$ has a 4-assignment $L_1$
that is $2_{in}$-forcing for $v$.  By symmetry, assume that $L_1(v)=1234$ and
$L_1$ forces $\vph(v)\in\{12,13\}$.  Let $B_2$ be the other block in $\B(G)$ and
let $B_2'$ consist of $B_2$ and the path from $B_2$ to $v$.  By
Lemma~\ref{C4-lem},
$B_2'$ has a 4-assignment $L_2$ that is $4_{out}$-forcing for $v$.  By permuting
colors in $L_2$, we can ensure that $L_2(v)=1234$ and that $L_2$ forces
$\vph(v)\notin\{12,13\}$.  Let $L=L_1\cup L_2$.  Since $G$ has no 2-fold
$L$-coloring, we conclude that $G$ is not $(4,2)$-choosable, a contradiction.
So no block of $G$ is $\theta_{1,2s+1,2t+1}$.

Suppose that both blocks in $\B(G)$ are strong subdivisions of $\theta_{2,2,2}$, and let
$v$ be a cut-vertex in one block, say $B_1$.  By Claim~\ref{clm1}, block $B_1$
has a 4-assignment $L_1$ that is $3_{in}$-forcing for $v$, and such that
$L_1(v)=1234$.  Further, by symmetry, we assume that $L_1$ forces
$\vph(v)\in\{12,13,14\}$.  Let $B_2$ denote the other block in $\B(G)$, and let
$B_2'$ consist of $B_2$ and the path from $B_2$ to $v$.  By Claim~\ref{clm1}
and Proposition~\ref{obs1}, subgraph $B_2'$ has a 4-assignment $L_2$ such that
$L_2(v)=1234$ and $L_2$ is $3_{out}$-forcing for $v$, requiring that
$\vph(v)\notin\{12,13,14\}$.  Now $L_1\cup L_2$ is a 4-assignment such that $G$
has no 2-fold $L$-coloring.  Thus, $G$ is not $(4,2)$-choosable, a contradiction.  

Finally, suppose that $B_1=C_{2s}$ and $B_2$ is a strong subdivision of
$\theta_{2,2,4}$.  By the Strong Minor Lemma, we assume that $B_2=\theta_{2,2,4}$.
Let $v$ be the cut-vertex in $B_2$, and let $B_1'$ consist of $B_1$ and the path
from $B_1$ to $v$.  By Lemma~\ref{C4-lem}, subgraph $B_1'$ has a 4-assignment $L_1$ that
is $4_{out}$-forcing for $v$; by symmetry we assume that $L_1(v)=1234$ and $L_1$
forces $\vph(v)\notin\{12,13\}$.  By Claim~\ref{clm4}, block $B_2$ has a
4-assignment $L_2$ that is $2_{in}$-forcing for $v$; by symmetry we assume that
$L_2(v)=1234$ and that $L_2$ forces $\vph(v)\in\{12,13\}$.  Let $L=L_1\cup L_2$.
Since $G$ has no 2-fold $L$-coloring, we conclude that $G$ is not
(4,2)-choosable, a contradiction.

Thus either $\B(G)=\{C_{2s},C_{2t}\}$ or $\B(G)=\{C_{2s},\theta_{2,2,2}\}$.
\end{proof}


\section{$\card{\B(G)}\ge3$}
\label{BG3-sec}

\begin{lem} Let $G$ be (4,2)-choosable with $\delta\ge 2$. 
Now $\card{\B(G)}\le 3$.  Further, if $\card{\B(G)}=3$, then each block of $G$
is a cycle, $G$ has a path $P$ that contains at least one edge in each block
of $\B(G)$, the block of $\B(G)$ that appears second along $P$ is $C_4$, and the
cut-vertices in this $C_4$ are non-adjacent (see Figure~\ref{figFF}).
\label{BG3-lem}
\end{lem}

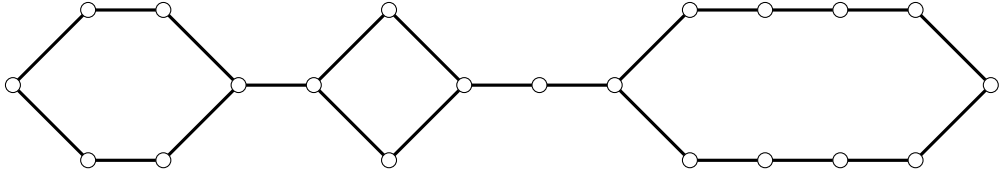
\begin{figure}
\centering
\begin{tikzpicture}[scale = 10]
\Vertex[style = unlabeledStyle, x = 0.150, y = 0.800, L = \small {}]{v0}
\Vertex[style = unlabeledStyle, x = 0.050, y = 0.700, L = \small {}]{v1}
\Vertex[style = unlabeledStyle, x = 0.150, y = 0.600, L = \small {}]{v2}
\Vertex[style = unlabeledStyle, x = 0.250, y = 0.600, L = \small {}]{v3}
\Vertex[style = unlabeledStyle, x = 0.350, y = 0.700, L = \small {}]{v4}
\Vertex[style = unlabeledStyle, x = 0.250, y = 0.800, L = \small {}]{v5}
\Vertex[style = unlabeledStyle, x = 0.450, y = 0.700, L = \small {}]{v6}
\Vertex[style = unlabeledStyle, x = 0.550, y = 0.800, L = \small {}]{v7}
\Vertex[style = unlabeledStyle, x = 0.650, y = 0.700, L = \small {}]{v8}
\Vertex[style = unlabeledStyle, x = 0.550, y = 0.600, L = \small {}]{v9}
\Vertex[style = unlabeledStyle, x = 0.750, y = 0.700, L = \small {}]{v10}
\Vertex[style = unlabeledStyle, x = 0.850, y = 0.700, L = \small {}]{v11}
\Vertex[style = unlabeledStyle, x = 0.950, y = 0.800, L = \small {}]{v12}
\Vertex[style = unlabeledStyle, x = 1.050, y = 0.800, L = \small {}]{v13}
\Vertex[style = unlabeledStyle, x = 1.150, y = 0.800, L = \small {}]{v14}
\Vertex[style = unlabeledStyle, x = 1.250, y = 0.800, L = \small {}]{v15}
\Vertex[style = unlabeledStyle, x = 1.350, y = 0.700, L = \small {}]{v16}
\Vertex[style = unlabeledStyle, x = 1.250, y = 0.600, L = \small {}]{v17}
\Vertex[style = unlabeledStyle, x = 1.150, y = 0.600, L = \small {}]{v18}
\Vertex[style = unlabeledStyle, x = 1.050, y = 0.600, L = \small {}]{v19}
\Vertex[style = unlabeledStyle, x = 0.950, y = 0.600, L = \small {}]{v20}
\Edge[](v15)(v16)
\Edge[](v17)(v16)
\Edge[](v13)(v14)
\Edge[](v15)(v14)
\Edge[](v17)(v18)
\Edge[](v19)(v18)
\Edge[](v11)(v20)
\Edge[](v19)(v20)
\Edge[](v11)(v12)
\Edge[](v13)(v12)
\Edge[](v8)(v10)
\Edge[](v11)(v10)
\Edge[](v7)(v8)
\Edge[](v9)(v8)
\Edge[](v7)(v6)
\Edge[](v9)(v6)
\Edge[](v5)(v4)
\Edge[](v6)(v4)
\Edge[](v3)(v4)
\Edge[](v1)(v0)
\Edge[](v5)(v0)
\Edge[](v3)(v2)
\Edge[](v2)(v1)
\end{tikzpicture}
\caption{An example of the case allowable in Lemma~\ref{BG3-lem}.\label{figFF}}
\end{figure}

\begin{proof}
By Lemma~\ref{BG2-lem}, we can assume that every block in $\B(G)$ is either an even
cycle or $\theta_{2,2,2}$ (since otherwise $G$ contains a subgraph that is not
(4,2)-choosable).  Suppose some $v\in V(G)$ is in at least three
distinct blocks of $G$.  Let $H'_1, H'_2, H'_3$ be three components of $G-v$ and
let $H_i$ be the subgraph induced by $\{v\}\cup V(H'_i)$, for each $i$.
Since $\delta\ge 2$, each $H_i$ contains a cycle with an incident path (possibly
length 0) ending at $v$.  By Lemma~\ref{C4-lem} and Proposition~\ref{obs1}, we
give each $H_i$ a 4-assignment $L_i$ that is $4_{out}$-forcing for $v$.
By permuting colors in these assignments, we get that $L_i(v)=1234$ for each $i$
and that assignment $L_1$ forces $\vph(v)\notin \{12,13\}$, assignment $L_2$
forces $\vph(v)\notin \{23,24\}$, and assignment $L_3$ forces $\vph(v)\notin
\{14,24\}$.  Thus, $G$ is not (4,2)-choosable, a contradiction.

Now suppose that $G$ contains a block $B_0$ with vertices $v_1, v_2, v_3$ and
each $v_i$ is in a block $B_i$ distinct from $B_0$.  Let $T$ be a tree in $B_0$
that contains $v_1,v_2,v_3$ and such that the only leaves of $T$ are in the set
$\{v_1,v_2,v_3\}$.  Delete from $B_0$ all edges except those of $T$; call the
resulting graph $G'$.  (Note that $\delta(G')\ge 2$.)  If $v_1, v_2, v_3$
appear on a path in $G'$, then (by symmetry) $v_2$ appears second along the
path, so $v_2$ is in at least three distinct blocks in $G'$.  Now $G'$ (and
hence $G$) is not (4,2)-choosable, as above.  So assume instead that $v_1, v_2,
v_3$ do not appear on a path in $G'$;
hence, there exists a vertex $w$ such that $v_1, v_2, v_3$ are in distinct
components of $G'-w$.  Thus, $w$ is in at least three distinct blocks in $G'$.
So again, $G'$ (and, thus, $G$) is not (4,2)-choosable.

The previous two paragraphs show that each vertex of $G$ is in at most two
blocks and that for each block at most two of its vertices appear in two blocks.
Thus, there exists a path $P$ that contains an edge of every block in $\B(G)$.

Suppose that $\card{\B(G)}\ge 4$.  By the Strong Minor Lemma, it suffices to
consider the case when each block $B_i\in \B(G)$ is a 4-cycle; assume they
appear in the order $B_1, B_2, B_3, B_4$ along $P$.  First, suppose that $B_2$
 (or $B_3$, by symmetry) has cut-vertices that are adjacent; call these vertices
$v$ and $w$.  Let $B_1'$ consist of $B_1$ and the path from $B_1$ to $B_2$; say
$v$ is the leaf in $B_1'$.  Similarly, let $B_3'$ denote $B_3$ and the path from
$B_3$ to $B_2$; note that $w$ is the leaf in $B_3'$.  By Lemma~\ref{C4-lem}, we give
$B_1'$ a 4-assignment that is $4_{out}$-forcing for $v$; by symmetry, assume
that $L_1(v)=1234$ and $L_1$ forces $\vph(v)\notin\{12,13\}$.  Similarly, we
give $B_2$ a 4-assignment $L_2$ with $L(v)=L(w)=1234$; by permuting colors in
$L_2$, we can require that $L_2$ forces $\vph(v)\notin\{14,24\}$.  Together
$L_1$ and $L_2$ force $\vph(v)\in\{23,34\}$.  Since $L_2(w)=L_2(v)$, also $L_1$
and $L_2$ force $\vph(w)\in\{14,12\}$.  Now we give $B_3'$ a $4_{out}$-forcing
assignment $L_3$ with $L_3(w)=1234$ and such that $L_3$ forces
$\vph(w)\notin\{12,14\}$.  Let $L=L_1\cup L_2\cup L_3$.  Since $G$ has no 2-fold
$L$-coloring, we conclude that $G$ is not (4,2)-choosable, a contradiction.

Now assume that the cut-vertices in $B_2$ (resp.~$B_3$) are non-adjacent; call
them $v_2$ and $w_2$ (resp.~$v_3$ and $w_3$).  Define $B_1'$ and $B_3'$ as
above, and define $B_4'$ analogously.  By Lemma~\ref{C4-lem}, we give a 4-assignment
to $B_1'\cup B_2$ that is $2_{comp}$-forcing for $v_2$.  This uses an assignment
$L_1$ for $B_1'$ that forces $\vph(v_2)\notin\{24,34\}$ and an assignment $L_2$
for $B_2$ that forces $\vph(v_2)\notin\{12,13\}$ (we also require
$L_1(v_2)=L_2(v_2)=1234$).  So, $L_1\cup L_2$ forces $\vph(v_2)\in \{14,23\}$.
It is easy to check that this forces $\vph(w_2)\in\{23,45\}$.  Thus, $L_1\cup
L_2$ is $2_{comp}$-forcing for $w_2$.  Similarly, we construct an
assignment for $B_3\cup B_4'$ that is $2_{comp}$-forcing for $v_3$.  By
Proposition~\ref{obs1}, we extend this to an assignment for $B_3'\cup B_4'$ that
is $2_{comp}$-forcing for $w_2$.  By permuting colors in the lists $B_3'\cup
B_4'$, we can ensure that $B_3'(w_2)=2345$ and that $B_3'\cup B_4$ forces
$\vph(w_2)\notin\{23,45\}$.  Thus, $G$ is not (4,2)-choosable, a contradiction.

Now assume that $\card{\B(G)}=3$ and that the second block along $P$, call
it $B_2$, is not $C_4$.  First suppose that $B_2$ is an even cycle of length at
least 6. If the two cut-vertices in $B_2$ are an even distance apart, then $G$
contains a strong minor of a graph in which a single vertex lies in three
edge-disjoint cycles.  As shown above, $G$ is not (4,2)-choosable.  If the two
cut-vertices in $B_2$ are an odd distance apart, then $G$ contains a strong minor
of a graph $G'$ with $\B(G')=\{C_4,C_4,C_4\}$, such that the middle $C_4$ has
cut-vertices that are adjacent.  Again, $G$ is not (4,2)-choosable, as shown
above.

Now assume that $B_2=\theta_{2,2,2}$.  Let $v$ and $w$ denote the cut-vertices in
$B_2$.  As shown above, we can assume $v$ and $w$ are non-adjacent (since every
pair of vertices in $\theta_{2,2,2}$ lie on a 4-cycle).  Define $B_1'$ and $B_3'$ as
above.  We give $B_2$ the 4-assignment $L_2$ shown in Figure~\ref{figXX}.
By Claim~\ref{clm1} of Lemma~\ref{BG2-lem}, assignment $L_2$ is
$3_{out}$-forcing for $v_1$ and $3_{in}$-forcing for $v_2$.  First suppose that
the cut-vertex $v$ is $v_1$ in
the figure, so $w$ is $w_1$.  By Lemma~\ref{C4-lem}, we give $B_1'$ a 4-assignment
$L_1$ that forces $\vph(v_1)\notin\{14,34\}$.  Similarly, we give $B_3'$ a
4-assignment $L_3$ that forces $\vph(w_1)\ne13$.  Let $L=L_1\cup L_2\cup L_3$.
Note that $G$ has no 2-fold $L$-coloring, as follows.  Assignment $L_2$ forces
$\vph(v_1)\in\{13,14,34\}$.  However, $L_1$ forces $\vph(v_1)\notin\{14,34\}$.
Thus, $\vph(v_1)=13$.  By Claim~\ref{clm1}, this implies that
$\vph(w_1)=13$.  But $L_3$ forces $\vph(w_1)\ne 13$, a contradiction.

Now instead assume that the cut-vertex $v$ is $v_2$ in Figure~\ref{figXX}, so
$w$ is $w_2$.  Now, similar to the previous case, we give $B_1'$ a 4-assignment
that forces $\vph(v_2)\notin\{12,23\}$.  Since $L_2$ forces $\vph(v_2)\in
\{12,23,24\}$, note that $L_1\cup L_2$ forces $\vph(v_2)=24$.  By
Claim~\ref{clm1}, this forces $\vph(w_2)=25$.  So we give $B_3'$ a 4-assignment
that forces $\vph(w_2)\ne 25$.  Thus, $G$ has no 2-fold $(L_1\cup L_2\cup
L_3)$-coloring, a contradiction.  This concludes the case that $B_2\ne C_4$.

Now suppose that $B_1$ (or $B_3$, by symmetry) is not an even cycle.  It
suffices to consider the case when $B_1$ and $B_2$ are 4-cycles and
$B_3=\theta_{2,2,2}$.  Define $B_1'$ and $B_3'$ as above.  Again, let $v$ and
$w$ be the cut-vertices in $B_2$, and $v$ be the leaf in $B_1'$.  Give $B_2$ a
4-assignment $L_2$ that forces $\vph(v)\notin\{12,13\}$.  Similarly, give $B_1'$
a 4-assignment $L_1$ that forces $\vph(v)\notin\{23,24\}$.  So $L_1\cup L_2$
forces $\vph(v)\in \{14,34\}$.  Let $x$ denote the neighbor of $v$ in $B_2$ with
$L_2(x)=L_2(v)$.  Now $\vph(v)\in\{14,34\}$ forces $\vph(x)\in\{23,12\}$.  Thus
$2\notin \vph(w)$.  Hence, $L_1\cup L_2$ is $3_{out}$-forcing for $w$.  Now the
first statement in Claim~\ref{clm1} (together with Proposition~\ref{obs1}),
allow us to give $B_3'$ a 4-assignment $L_3$ that is $3_{in}$-forcing for $w$.
By permuting colors in $L_3$, we can require that $L_3(w)=2345$ and $L_3$ forces
$2\in \vph(w)$.  Thus, $G$ has no 2-fold $(L_1\cup L_2\cup L_3)$-coloring, which
is a contradiction.  This concludes the proof.
\end{proof}

\begin{thm}
\label{necessity-thm}
If a connected graph is (4,2)-choosable,
then either its core is one of the following six types (where $s$ and $t$ are
positive integers):
(i) $K_1$,
(ii) $C_{2s}$,
(iii) $\theta_{2,2s,2t}$, 
(iv) $\theta_{1,2s+1,2t+1}$,
(v) $K_{2,4}$,
(vi) a graph formed from $K_{3,3}-e$ by subdividing a single edge incident to a
vertex of degree 2 an even number of times, or else 
(vii) $\B(G)=\{C_{2s}, C_{2t}\}$,
(viii) $\B(G)=\{\theta_{2,2,2},C_{2s}\}$, or
(ix) $\B(G)=\{C_4, C_{2s}, C_{2t}\}$, where the $C_4$ appears second on a path
passing through all three blocks, and the two cut-vertices in $C_4$ are
non-adjacent. 
\end{thm}
\begin{proof}
The proof is simply collecting our results thus far.  If we are not in (i), then
the core of $G$ has minimum degree at least 2.  When $G$ is (4,2)-choosable
and 2-connected, we are in (i-vi) by Theorem~\ref{BG1-lem}.  When
$\card{\B(G)}=2$, we are in (vii) or (viii) by Lemma~\ref{BG2-lem}.  When
$\card{\B(G)}\ge 3$, we are in (ix) by Lemma~\ref{BG3-lem}.
\end{proof}

Theorem~\ref{necessity-thm} confirms one direction of the characterization
conjectured by Meng, Puleo, and Zhu.  If a graph is (4,2)-choosable, then it is
a graph they conjectured was (4,2)-choosable.  To complete the proof of their
conjecture, we must show that each graph in (i)--(ix) of
Theorem~\ref{necessity-thm} is indeed (4,2)-choosable.  Case (i) is trivial, and 
Tuza and Voigt~\cite{TV} handled (ii) and (v).  Meng, Puleo, and Zhu~\cite{MPZ}
handled (iii), (iv), and (vii).  So in the next section we must handle cases
(vi), (viii), and (ix).

\section{Graphs that are (4,2)-choosable}
In this section we complete the proof of Conjecture~\ref{main-conj}, by showing
that every graph in cases (vi), (viii), and (ix) of Conjecture~\ref{main-conj}
(and Theorem~\ref{necessity-thm}) is indeed (4,2)-choosable.  We should mention
now that in one case the proof is computer-aided.  The main point of this
section is to show that verifying (4,2)-choosability for each graph in the four
infinite families (case (viii) contains two of these families) can actually be
reduced to verifying (4,2)-choosability of four specific graphs, three with 8
vertices and one with 10 vertices; these graphs are shown in Figure~\ref{figGG}.  
For all of these graphs, Meng, Puleo, and
Zhu already verified (4,2)-choosability by computer.  

For the three of these graphs with cut-vertices we sketch how to check
(4,2)-choosability by hand.  For the fourth, it seem we really need computer
case-checking.  To sketch these proofs we need two ideas.  The first was
foreshadowed by our notion of $k$-forcing.  We show that each 4-assignment 
for $C_4$ forbids at most two colorings of each vertex; likewise, each
4-assignment for $\Theta_{2,2,2}$ forbids at most three colorings of each vertex.
Thus, for the two graphs on the left in Figure~\ref{figGG}, each 4-assignment
admits some coloring of the cut-vertex that extends to both blocks.
A similar, but more subtle, argument works for the 10-vertex graph.

To prove the claims in the previous paragraph about 4-assignments for $C_4$ and
$\Theta_{2,2,2}$, we introduce the notion of \emph{flat} list assignments (see
Definition~\ref{defn2}).
Intuitively, these are list assignments that are the most difficult to color from.
More formally, we show that each 4-assignment $L$ can be mapped to a flat
4-assignment $L'$ such that the 2-fold $L'$-colorings map injectively into
2-fold $L$-colorings.  The numbers of flat 4-assignments for every graph is
finite (4 for $C_4$, and 35 for $\Theta_{2,2,2}$), so we can prove the claims by
simply examining each such 4-assignment.

For the 2-connected graph, as a double-check, Landon
Rabern independently wrote a computer program to verify that it is 
(4,2)-choosable.  The biggest theoretical insight in this process was showing
that we only need to consider list assignments with at most 8 colors in total;
we prove this in Lemma~\ref{SPLish}.
Another important step was generating the list assignments to avoid isomorphic
copies (one list assignment is formed from another by permuting the names of
colors).  Otherwise, his program was essentially brute force.

As we mentioned earlier, Meng, Puleo, and Zhu proved that a graph is
(4,2)-choosable whenever it is formed from two vertex disjoint cycles by adding
a path from a vertex on one cycle to a vertex on the other.  We 
use a corollary of this fact, so we include their lemma next. 

\begin{lem}
If $G$ is a connected graph with $\B(G)=\{C_{2s},C_{2t}\}$, then $G$ is
(4,2)-choosable.
\label{barbell-lem}
\end{lem}

\begin{cor}
If $G$ is a connected graph with $\B(G)=\{C_{2s}\}$, then for each vertex $v\in
V(G)$, there does not exist a 4-assignment $L$ that is $3$-forcing for $v$.
\label{barbell-cor}
\end{cor}
\begin{proof}
Suppose the corollary is false; let $G$, $v$, and $L$ be a counterexample.
Let $C$ denote the cycle in $G$, and let $G'$ be the subgraph of $G$ consisting
of $C$, $v$, and the path from $v$ to $C$.  By symmetry, we assume that
$L(v)=1234$ and that $L$ forces either (i) $\vph(v)\in\{12,13,23\}$, (ii)
$\vph(v)\in \{12,13,14\}$, or (iii) $\vph(v)\in\{12,23,34\}$.  In cases (i) and
(ii), we proceed as follows.  Form $H$ from two copies of $G'$ by adding an edge
between the copies of $v$ (with each vertex keeping its list from $L$).  Now $H$
has no 2-fold $L$-coloring, contradicting Lemma~\ref{barbell-lem}.  In case (iii),
form two copies of $G'$ and $L$, but in one copy permute the colors in the lists
as follows: $1\to 3$, $2\to 1$, $3\to 4$, $4\to 2$.  Now form $H$ from these two
copies by identifying their copies of $v$.  The original list assignment $L$
forces $\vph(v)\in\{12,23,34\}$, but the modified version of $L$ forces
$\vph(v)\in\{13,14,24\}$.  Thus, $H$ has no coloring from this 4-assignment,
again contradicting Lemma~\ref{barbell-lem}.
\end{proof}

\begin{lem}
Let $G$ be a graph with $\B(G)=\{\theta_{2,2,2},C_{2s}\}$ or $\B(G)=\{C_4,C_{2s},
C_{2t}\}$ and let $e$ be an edge of $G$
not in any cycle.  Form $G'$ from $G$ by contracting $e$.  If $G$ is
not (4,2)-choosable, then neither is $G'$.
\label{edge-contract-lem}
\end{lem}
\begin{proof}
We handle explicitly the case $\B(G)=\{\theta_{2,2,2},C_{2s}\}$; the case 
$\B(G)=\{C_4,C_{2s},C_{2t}\}$ is nearly identical, so we omit it.
Let $G$, $e$, and $G'$ be as in the statement of the lemma.  Let $v$ and $w$ be the
endpoints of $e$, with $v$ closer to the $\theta_{2,2,2}$ and $w$ closer to the
$C_{2s}$.  Without loss of generality, we can assume that $v\in
V(\theta_{2,2,2})$.  Let $B_1=\theta_{2,2,2}$ and let $B'_2$ consist of $C_{2s}$
and the path from it to $v$.  Suppose $G$ is not (4,2)-choosable and let $L$ be
a 4-assignment showing this.  By Corollary~\ref{barbell-cor}, $L$ restricted to
$B_1$ must be 2-forcing for $v$ (otherwise $v$ has a coloring that extends to
both $B_1$ and $B_2'$, so $G$ is 2-fold $L$-colorable).
We show that if $L$ restricted to $B'_2$ is $4_{comp}$-forcing
(resp.~$4_{out}$-forcing) for $v$, then it is also $4_{comp}$-forcing
(resp.~$4_{out}$-forcing) for $w$.  Thus, by permuting colors on $B'_2-v$ and
identifying $w$ with $v$ in $\theta_{2,2,2}$, we get a 4-assignment $L'$ for
$G'$ witnessing that $G'$ is not (4,2)-choosable.

Suppose $L$ is $4_{comp}$-forcing for $v$.  By symmetry, assume 
$L(v)=1234$ and $L$ restricted to $B'_2$ forces $\vph(v)\notin\{12,34\}$.
Suppose $12\not\subseteq L(w)$.  Now
$\card{L(w)\setminus 12}\ge 3$, so $\vph(v)=12$ can be extended to $w$ in at
least three ways.  By Corollary~\ref{barbell-cor}, at least one of these three
ways is not forbidden by $L$ on $B_2'-v$.  Thus, $L$ does not force
$\vph(v)\ne12$, a contradiction.  So $12\subseteq L(w)$.  Similarly,
$34\subseteq L(w)$.  So $L(w)=1234$.  Since $L$ forces $\vph(v)\ne 12$,
clearly $L$ forces $\vph(w)\ne 34$.  Since $L$ forces $\vph(v)\ne 34$, also
$L$ forces $\vph(w)\ne 12$.  Thus, $L$ forces $\vph(w)\notin\{12,34\}$.  Hence,
$L$ is $4_{comp}$-forcing on $w$, as desired.  If instead $L$ is
$4_{out}$-forcing for $v$, the proof that $L$ is
$4_{out}$-forcing for $w$ is very similar.  By symmetry, we assume 
$L(v)=1234$ and $L$ forces $\vph(v)\notin\{12,13\}$.  This implies
$L(w)=123\alpha$ and $L$ forces $L(w)\notin\{2\alpha,3\alpha\}$.  
By permuting colors on $B_2'-v$, we get lists on $G'$ such that their
restriction to the contraction of $B_2'$ forces $\vph(v)\notin\{12,13\}$ or
$\vph(v)\notin\{12,34\}$, just like the lists on $B_2'$ did.  So if $G$ is not
(4,2)-choosable, then neither is $G'$.
\end{proof}

We need one more lemma of Meng, Puleo, and Zhu (it is their Lemma~7.4).

\begin{lem}
Let $G$ be a graph containing a path $P$ on 5 vertices which all have degree 2
in $G$, and let $G'$ be the graph obtained by deleting the middle vertex of $P$
and merging its neighbors.  The original graph $G$ is $(4m,2m)$-choosable if and
only if the merged graph $G'$ is $(4m,2m)$-choosable.
\label{path-contract-lem}
\end{lem}

\begin{lem}
To verify Conjecture~\ref{main-conj}, it suffices to show that the four graphs
in Figure~\ref{figGG} are (4,2)-choosable.
\label{just-four-lem}
\end{lem}
\begin{proof}
By Theorem~\ref{necessity-thm} and the final paragraph of Section~\ref{BG3-sec},
to prove the conjecture it suffices to show
that all graphs in its cases (vi), (viii), and (ix) are (4,2)-choosable.
By Lemma~\ref{path-contract-lem} we assume every block $C_{2s}$ or
$C_{2t}$ is in fact $C_4$, and that every path in case (vi) has length at most
4.  By Lemma~\ref{edge-contract-lem}, we assume every edge of $G$ is
in a cycle.  Finally, the instance of case (vi) where the path of unspecified
even length has length 2 is a strong minor of the case where the path has length
4.  Thus, we need only consider the latter, which is shown in Figure~\ref{figGG}.
\end{proof}
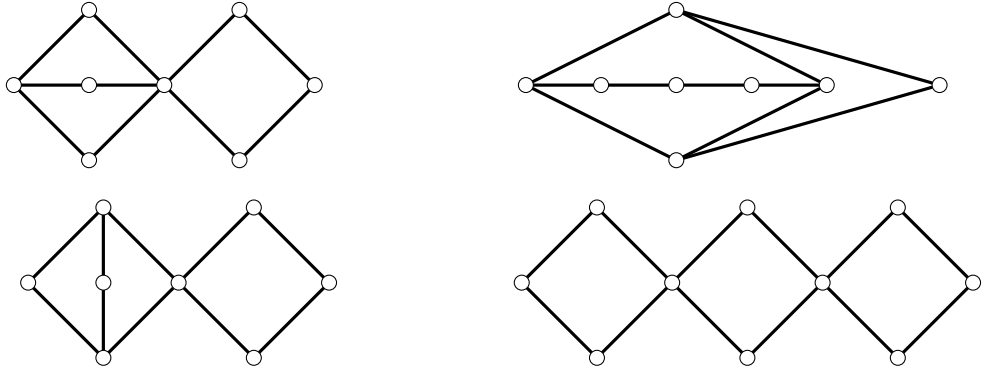
\begin{figure}
\centering
\begin{minipage}[b]{0.45\linewidth}
\centering
\begin{tikzpicture}[scale = 10]
\Vertex[style = unlabeledStyle, x = 0.300, y = 0.850, L = \small {}]{v0}
\Vertex[style = unlabeledStyle, x = 0.200, y = 0.750, L = \small {}]{v1}
\Vertex[style = unlabeledStyle, x = 0.300, y = 0.650, L = \small {}]{v2}
\Vertex[style = unlabeledStyle, x = 0.400, y = 0.750, L = \small {}]{v3}
\Vertex[style = unlabeledStyle, x = 0.500, y = 0.850, L = \small {}]{v4}
\Vertex[style = unlabeledStyle, x = 0.500, y = 0.650, L = \small {}]{v5}
\Vertex[style = unlabeledStyle, x = 0.600, y = 0.750, L = \small {}]{v6}
\Vertex[style = unlabeledStyle, x = 0.300, y = 0.750, L = \small {}]{v7}
\Edge[](v0)(v1)
\Edge[](v0)(v3)
\Edge[](v2)(v1)
\Edge[](v2)(v3)
\Edge[](v4)(v3)
\Edge[](v4)(v6)
\Edge[](v5)(v3)
\Edge[](v5)(v6)
\Edge[](v1)(v7)
\Edge[](v3)(v7)
\end{tikzpicture}
\end{minipage}
\begin{minipage}[b]{0.45\linewidth}
\centering
\begin{tikzpicture}[scale = 10]
\Vertex[style = unlabeledStyle, x = 0.200, y = 0.750, L = \small {}]{v0}
\Vertex[style = unlabeledStyle, x = 0.400, y = 0.750, L = \small {}]{v1}
\Vertex[style = unlabeledStyle, x = 0.600, y = 0.750, L = \small {}]{v2}
\Vertex[style = unlabeledStyle, x = 0.300, y = 0.750, L = \small {}]{v3}
\Vertex[style = unlabeledStyle, x = 0.500, y = 0.750, L = \small {}]{v4}
\Vertex[style = unlabeledStyle, x = 0.400, y = 0.650, L = \small {}]{v5}
\Vertex[style = unlabeledStyle, x = 0.400, y = 0.850, L = \small {}]{v6}
\Vertex[style = unlabeledStyle, x = 0.750, y = 0.750, L = \small {}]{v7}
\Edge[](v1)(v4)
\Edge[](v2)(v4)
\Edge[](v0)(v3)
\Edge[](v1)(v3)
\Edge[](v0)(v5)
\Edge[](v2)(v5)
\Edge[](v0)(v6)
\Edge[](v2)(v6)
\Edge[](v7)(v6)
\Edge[](v7)(v5)
\end{tikzpicture}
\end{minipage}
~
\vspace{.15in}

\begin{minipage}[b]{0.45\linewidth}
\centering
\begin{tikzpicture}[scale = 10]
\Vertex[style = unlabeledStyle, x = 0.300, y = 0.850, L = \small {}]{v0}
\Vertex[style = unlabeledStyle, x = 0.200, y = 0.750, L = \small {}]{v1}
\Vertex[style = unlabeledStyle, x = 0.300, y = 0.650, L = \small {}]{v2}
\Vertex[style = unlabeledStyle, x = 0.400, y = 0.750, L = \small {}]{v3}
\Vertex[style = unlabeledStyle, x = 0.500, y = 0.850, L = \small {}]{v4}
\Vertex[style = unlabeledStyle, x = 0.500, y = 0.650, L = \small {}]{v5}
\Vertex[style = unlabeledStyle, x = 0.600, y = 0.750, L = \small {}]{v6}
\Vertex[style = unlabeledStyle, x = 0.300, y = 0.750, L = \small {}]{v7}
\Edge[](v0)(v1)
\Edge[](v0)(v3)
\Edge[](v2)(v1)
\Edge[](v2)(v3)
\Edge[](v4)(v3)
\Edge[](v4)(v6)
\Edge[](v5)(v3)
\Edge[](v5)(v6)
\Edge[](v0)(v7)
\Edge[](v2)(v7)
\end{tikzpicture}
\end{minipage}
\begin{minipage}[b]{0.45\linewidth}
\centering
\begin{tikzpicture}[scale = 10]
\Vertex[style = unlabeledStyle, x = 0.300, y = 0.850, L = \small {}]{v0}
\Vertex[style = unlabeledStyle, x = 0.200, y = 0.750, L = \small {}]{v1}
\Vertex[style = unlabeledStyle, x = 0.300, y = 0.650, L = \small {}]{v2}
\Vertex[style = unlabeledStyle, x = 0.400, y = 0.750, L = \small {}]{v3}
\Vertex[style = unlabeledStyle, x = 0.500, y = 0.850, L = \small {}]{v4}
\Vertex[style = unlabeledStyle, x = 0.500, y = 0.650, L = \small {}]{v5}
\Vertex[style = unlabeledStyle, x = 0.600, y = 0.750, L = \small {}]{v6}
\Vertex[style = unlabeledStyle, x = 0.700, y = 0.850, L = \small {}]{v7}
\Vertex[style = unlabeledStyle, x = 0.800, y = 0.750, L = \small {}]{v8}
\Vertex[style = unlabeledStyle, x = 0.700, y = 0.650, L = \small {}]{v9}
\Edge[](v7)(v8)
\Edge[](v9)(v8)
\Edge[](v7)(v6)
\Edge[](v9)(v6)
\Edge[](v4)(v6)
\Edge[](v5)(v6)
\Edge[](v4)(v3)
\Edge[](v5)(v3)
\Edge[](v0)(v3)
\Edge[](v2)(v3)
\Edge[](v0)(v1)
\Edge[](v2)(v1)
\end{tikzpicture}
\end{minipage}
\caption{To check that the four infinite families of graphs in cases (vi),
(viii), and (ix) of Conjecture~\ref{main-conj} are (4,2)-choosable, it suffices
to check that the four graphs shown above are (4,2)-choosable.\label{figGG}}
\end{figure}

Next we show that for each graph $G$ in Figure~\ref{figGG}, if $G$ is
not (4,2)-choosable, then this is witnessed by some 4-assignment $L$ such that
$\card{\cup_{v\in V(G)}L(v)}\le 8$.  
This observation was useful in restricting
the list assignments that Rabern's program needed to consider.  For a graph
$G$, and list assignment $L$, let \emph{$\pot(L)$}\aside{$\pot(L)$} denote $\cup_{v\in
V(G)}L(v)$.  

\begin{lem}Let $G$ be any graph in Figure~\ref{figGG}. If $G$ is not (4,2)-choosable,
then this is witnessed by some 4-assignment $L$ with $\card{\pot(L)}\le 8$.
\label{SPLish}
\end{lem}
\begin{proof}
Suppose the lemma is false, and let $G$ and $L$ be a counterexample.  So
$\card{\pot(L)}>8$.  
Now $V(G)$ contains a subset $\{x_1,x_2\}$, call it $X$, such that
$G-X$ consists of vertex disjoint paths, each with at most three vertices.
Further, if a path $P$ has exactly three vertices, then the interior vertex of
$P$ has no neighbors outside of $P$.
Pick $S\subseteq \pot(L)$ such that $L(x_1)\cup
L(x_2)\subseteq S$ and $\card{S}=8$.  
We will find a 4-assignment $L'$ such that $\pot(L')\subseteq S$ and $G$ has no
2-fold $L'$-coloring.

If any color $\alpha\in \pot(L)$ has a component of $G_{\alpha}$ that
is an isolated vertex $v$, and $v\notin X$, then replace $\alpha$ in $L(v)$ with a color in a
list for some neighbor $w$ of $v$.  
Repeat this step until no such $\alpha$ exists.  (Eventually this happens, since
at each step we decrease the sum of the numbers of components of $G_\alpha$,
taken over all $\alpha\in \pot(L)$.)  Now if there exists $\alpha\in
\pot(L)\setminus S$ such that $G_\alpha$ has a component $C$ of size at most 2,
then we replace $\alpha$ in $L(v)$, for each $v\in C$, with a color $\beta\in
S\setminus \cup_{v\in C}L(v)$.  (This is possible since $|\cup_{v\in C}L(v)|\le
2(4)-1<|S|$.)  Again, repeat this step until it is no longer possible.
Finally, suppose there exists $\alpha\in \pot(L)\setminus S$ such
that $G_\alpha$ contains a component $C$ (which is a path) on three vertices;
let $w$ denote the center
vertex of this path.  Because of the first sentence in this paragraph, each
color in $L(w)$ appears in $L(v)$ for some neighbor $v$ of $w$.  Thus,
$|\cup_{x\in C}L(x)|\le 2(4)-1<|S|$.  So we can again replace $\alpha$ with some
color $\beta\in (S\setminus\cup_{x\in C}L(x))$.  By repeating this process, we
eventually reach a list assignment $L'$ such that $G$ has a 2-fold $L'$-coloring
only if $G$ has a 2-fold $L$-coloring; further $\pot(L')\subseteq S$, so
$|\pot(L')|\le 8$.
\end{proof}

\begin{proof}[Proof of Main Theorem]
In view of Lemma~\ref{just-four-lem}, it suffices to show that the four graphs
in Figure~\ref{figGG} are (4,2)-choosable.  For the three graphs with a
cut-vertex, a proof is sketched in Theorem~\ref{cut-vert-thm} below (and it was
also shown by a computer program of Meng, Puleo, and Zhu).  For the 2-connected
graph, this was checked by a program written by Landon Rabern (and also the
program of Meng, Puleo, and Zhu).  This completes the proof.
\end{proof}

\begin{defn}
\label{defn2}
For a list assignment $L$, and a color $\alpha\in \pot(L)$, let
$G_{\alpha}$ denote the subgraph of $G$ induced by vertices with $\alpha$ in
their list.  Fix a graph $G$, a list assignment $L$ for $G$, and $\alpha\in
\pot(L)$.  Fix a component $C$ of $G_\alpha$ and a color $\beta\in
\pot(L)\setminus \cup_{v\in V(C)}L(v)$.  A \Emph{flattening move} for $L$
consists of replacing $\alpha$ by $\beta$ in $L(v)$ for each $v\in V(C)$.  Let
$\#(G)$ denote the number of components in $G$.  A list assignment $L$ is
\Emph{flat} if, among all list assignments $L'$ that can be formed from $L$ by
a sequence of flattening moves, $L$ minimizes $|\pot(L')|$ and, subject to
that, minimizes $\sum_{\alpha\in \pot(L')}\#(G_\alpha)$. 
When $L$ is flat, we typically assume $\pot(L)=\{1,\ldots,|\pot(L)|\}$.
For a graph $G$, let $f_k(G,i)$\aside{$f_k(G,i)$} 
denote the number of flat $k$-assignments $L$ for $G$ with $|\pot(L)|=i$ up to
isomorphism (this includes renaming colors, as well as automorphisms of $G$).
\end{defn}

\begin{lem}
Let $L$ and $L'$ be list assignments for $G$.  If $L'$ is formed from $L$ by a
sequence of flattening moves, then $b$-fold $L'$-colorings of $G$ map
injectively to $b$-fold $L$-colorings of $G$.  
\label{flat-lem}
\end{lem}
\begin{proof}
Let $G$, $L$, and $L'$ satisfy the hypotheses.
Our proof is by induction on the number of flattening moves $t$ in the sequence.  
The base case, $t=0$, is trivial.  Since a composition of injections is an
injection, it suffices to consider the case when $t=1$.  Suppose that $L'$ is
formed from $L$ by replacing $\alpha$ by $\beta$ in $L(v)$ for each $v\in V(C)$,
where $C$ is some component of $G_\alpha$.
To map a $b$-fold $L'$ coloring $\vph'$ to a $b$-fold $L$-coloring, for each
$v\in V(C)$ with $\beta\in \vph'(v)$, we replace $\beta$ by $\alpha$.
Clearly this mapping is an injection, since $\alpha\notin \cup_{v\in
V(C)}L'(v)$.  Further, the image is a $b$-fold $L$-coloring, since neither
$\alpha$ nor $\beta$ is used on any neighbor of $v$ in a $b$-fold
$L'$-coloring.
\end{proof}

\begin{cor}
\label{flat-cor}
To show that $G$ is $b$-fold $a$-choosable, it suffices to show that $G$ has a
$b$-fold $L$-coloring for every flat $a$-assignment $L$.
\end{cor}

\begin{thm}
The three graphs in Figure~\ref{figGG} with a cut-vertex are each
(4,2)-choosable.
\label{cut-vert-thm}
\end{thm}

The proof is tedious, but straightforward given the seven claims below, so we
omit many details and simply list the claims. 
Claims~\ref{flat-pot6-lem} and~\ref{flat-pot7-lem} are proved by case
analysis, showing that for each 4-assignment other than those listed in the
appendix, we can perform a flattening move.  The proof of
Claim~\ref{flat-pot6-lem} is fairly short, but that of
Claim~\ref{flat-pot7-lem} is longer.  The arguments rely heavily on the
automorphisms of $\Theta_{2,2,2}$.  To emphasize this, we use the notation
$K_{2,3}$, rather than $\Theta_{2,2,2}$.  

\begin{clm}
\label{flat-pot6-lem}
Every flat 4-assignment $L$ for $C_4$ has $|\pot(L)|\le 6$.
Further, $f_4(C_4,6)=1$, $f_4(C_4,5)=2$, and $f_4(C_4,4)=1$.
\end{clm}

\begin{clm}
Let $G=C_4$ and fix $v\in V(G)$.  No 4-fold
assignment for $G$ is $3$-forcing for $v$.
\label{C4-forcing-lem}
Let $L$ be a 4-assignment for $C_4$.  If $L$ is 4-forcing for some vertex $v\in
V(C_4)$, then the two colorings of $v$ that $L$ forbids share a common color.
\label{strong-C4-lem}
\end{clm}
By Corollary~\ref{flat-cor}, it suffices to prove each statement when $L$ is flat.
Claim~\ref{flat-pot6-lem} implies that we only need to check the four
4-assignments constructed in its proof.  
\sqed

\begin{clm}
\label{strong-C4-lem2}
Let $G=C_4$, with vertices $v_1,\ldots,v_4$ around the cycle.  Fix a
4-assignment $L$; in addition, forbid two colorings of $v_1$ that share a common
color.  Let $S$ be the set of $L$-colorings of $G$ that avoid the two forbidden
colorings of $v_1$.  Either (i) colorings in $S$ restrict to at least 3 distinct
colorings of $v_3$ or (ii) colorings in $S$ restrict to two colorings of $v_3$
that use disjoint color sets.
\end{clm}
By Corollary~\ref{flat-cor}, it suffices to prove each statement when $L$ is flat.
Claim~\ref{flat-pot6-lem} implies that we only need to check the four
4-assignments $L$ constructed in its proof; for each choice of $L$ we have four
choices for $v_1$ and twelve choices of pairs of colorings to forbid.
\sqed
%

\begin{clm}
The 10-vertex graph is (4,2)-choosable.
\end{clm}
This follows directly from Claims~\ref{strong-C4-lem} and~\ref{strong-C4-lem2}.
\sqed

\begin{clm}
\label{flat-pot7-lem}
Every flat 4-assignment $L$ for $K_{2,3}$ has $|\pot(L)|\le 7$.
Further, $f_4(K_{2,3},7)=1$, $f_4(K_{2,3},6)=27$, $f_4(K_{2,3},5)=6$, and
$f_4(K_{2,3},4)=1$.
\end{clm}

\begin{clm}
Let $G=K_{2,3}$ and fix $v\in V(G)$.  No 4-fold assignment for $G$ is $2$-forcing for $v$.
\label{K23-forcing-lem}
\end{clm}
By Corollary~\ref{flat-cor}, it suffices to prove the statement when $L$ is flat.
Now Lemma~\ref{flat-pot7-lem} implies that it is enough to check the 35
4-assignments $L$ constructed in its proof.
\sqed

\begin{clm}
If $G$ is a graph formed from $K_{2,3}$ and $C_4$ by identifying one vertex of
each, then $G$ is (4,2)-choosable.
\end{clm}
Let $G_1=K_{2,3}$ and $G_2=C_4$.
Fix $v_1\in V(G_1)$ and $v_2\in V(G_2)$.  Form $G$ by identifying $v_1$ and
$v_2$; denote the new vertex $v$.  Fix a 4-assignment $L$ for $G$.  Let $L_1$
and $L_2$ denote the restrictions (respectively) of $L$ to $G_1$ and $G_2$.
By Claim~\ref{K23-forcing-lem}, list assignment $L_1$ forbids at most 3
colorings of $v_1$.  By Claim~\ref{C4-forcing-lem}, list assignment $L_2$
forbids at most 2 colorings of $v_2$.  Thus, some coloring of $v$ extends to
both $G_1$ and $G_2$, which shows that $G$ has a 2-fold $L$-coloring.
\sqed

\section{$(2m,m)$-choosability for general $m$}
\label{sec:general-m}

In the introduction we mentioned that for every integer $m$ there exist
bipartite graphs that are not $(2m,m)$-choosable.  Below we give a construction
illustrating this.

\begin{prop}
\label{construction-prop}
Fix a positive integer $m$.  There exists a bipartite graph that is not
$(2m,m)$-choosable.
\end{prop}
\begin{proof}
Let $C$ be a 4-cycle with vertices $v_1,v_2,v_3,v_4$.  Let
$L(v_1)=L(v_2)=\{1,\ldots,2m\}$.  Let $L(v_4)=\{1,\ldots,2m-2,2m-1,2m+1\}$, and
let $L(v_3)=\{2,\ldots,2m+1\}$.  It is easy to check that $C$ has no $m$-fold
$L$-coloring $\vph$ with $\vph(v_1)=\{1,\ldots,m\}$.  By permuting the color
classes on $v_2,v_3,v_4$, for any $S\subset \{1,\ldots,2m\}$ with $\card{S}=m$,
we can construct a $2m$-assignment $L_S$ such that $L_S(v_1)=\{1,\ldots,2m\}$
but $C$ has no $m$-fold $L_S$-coloring $\vph_S$ with $\vph_S(v_1)=S$.  We begin
with ${2m\choose m}$ disjoint 4-cycles and for each $m$-element subset $S$ of
$\{1,\ldots,2m\}$ assign to some 4-cycle the list assignment $L_S$.  To form
$G$, we identify the copies of $v_1$ in all 4-cycles, with each vertex in the new
graph inheriting its list assignment from its original 4-cycle; call this list
assignment $L^*$, and note that $L^*$ is a $2m$-assignment.  Clearly, the
resulting graph $G$ has no $m$-fold $L^*$-coloring.  Thus, $G$ is not
$(2m,m)$-choosable.
\end{proof}

Many of the ideas we used to characterize (4,2)-choosable graphs apply more
generally.
Tuza and Voigt~\cite{TV} used Rubin's characterization of (2,1)-choosable graphs
to prove that they are $(2m,m)$-choosable for every $m$.  When $m$ is odd, the
characterization of $(2m,m)$-choosable graphs is simple: they are precisely the
(2,1)-choosable graphs.  To prove this, we need to show that every
$(2m,m)$-choosable graph is (2,1)-choosable.  This result is generally
attributed to Voigt~\cite{voigt98}, although her manuscript does not seem to
have been published, and is not widely available.  So, for completeness, we
include a short argument of Gutner and Tarsi~\cite{GT}.

\begin{lem}
If $m$ is odd and $G$ is $(2m,m)$-choosable, then $G$ is (2,1)-choosable.
\label{GT-lem}
\end{lem}
\begin{proof}
Let $G$ be $(2m,m)$-choosable, for some odd $m$.  Let $L$ be a 2-assignment
to $G$.  We show that $G$ is $L$-colorable.  Suppose
$L(v)=\{\alpha,\beta\}$.  Let
$L'(v)=\{\alpha_1,\ldots,\alpha_m,\beta_1,\ldots,\beta_m\}$.  For each vertex
$w\in V(G)$ construct $L'(w)$ from $L(w)$ analogously.  Since $L'$ is a
$2m$-assignment, by hypothesis $G$ has an $m$-fold $L'$-coloring
$\vph'$.  If $\card{\vph'(v)\cap\{\alpha_1,\ldots,\alpha_m\}}>\frac{m}2$, then
let $\vph(v)=\alpha$; otherwise, let $\vph(v)=\beta$.  Analogously, for each
vertex $w\in V(G)$, let $\vph(w)$ be the color in $L(w)$ that appears most often
(with subscripts) in $\vph'(w)$.  Since $\vph'$ is proper, so is $\vph$.
(The key observation is that if
$\card{\vph'(v)\cap\{\alpha_1,\ldots,\alpha_m\}}>\frac{m}2$ and
$\card{\vph'(w)\cap\{\alpha_1,\ldots,\alpha_m\}}>\frac{m}2$, then
$\vph'(v)\cap \vph'(w)\ne \emptyset$, so $v$ and $w$ are non-adjacent.)
\end{proof}

We have not said much about algorithms thus far.  So it is worth mentioning
that if we fix $m$ and also bound the number of vertices of degree at least 3 in
an input graph $G$, then given a $2m$-assignment $L$ for $G$, we can test in
linear time whether $G$ has an $m$-fold $L$-coloring.

\begin{thm}
Fix positive integers $m$ and $C$.  Let $G$ and $L$ be given, where $G$ is a
graph with at most $C$ vertices of degree at least 3 and $L$ is a
$2m$-assignment.  We can check in linear time whether $G$ has an $m$-fold
$L$-coloring.
\end{thm}
\begin{proof}
Let $X$ be the set of vertices of degree at least 3 and $G'=G\setminus X$.  Note
that $G'$ is a disjoint union of paths.  For each path $P$, we compute the
$m$-fold $L$-colorings of its endpoints that extend to an $m$-fold $L$-coloring
of $P$, as follows.

Let $V(P)=\{v_1,\ldots,v_k\}$.  For each $i\ge 2$, we compute the $m$-fold
$L$-colorings of $v_1$ and $v_i$ that extend to $G[v_1,\ldots,v_i]$.  Having
solved the problem for $v_1$ and $v_{i-1}$, we can solve it for
$v_1$ and $v_i$ in constant time.  Thus, we can solve the problem for the
endpoints of each path $P$ in time linear in the length of $P$.  
Now we consider all at most ${2m\choose m}^C$ $m$-fold $L$-colorings of $X$ and
check which of these (if any) extend to all paths in $G'$.
Since $m$ and $C$ are constants, so is ${2m\choose m}^C$.  Hence, it suffices
to check whether a single $m$-fold $L$-coloring $\vph$ of $X$
extends to $G'$

Fix an $m$-fold $L$-coloring of $X$.  Having preprocessed each path $P$ in $G'$
as above, we can check whether $\vph$ extends to $P$ in constant time (searching
through the constant-sized results of the preprocessing).  Thus, we can check
whether $\vph$ extends to $G'$ in time linear in the number of paths in $G'$
(which is at most linear in the order of $G$).

Finally,
suppose $G$ has an $m$-fold $L$-coloring and we want to construct the coloring.
For each path $P$, for each $m$-fold $L$-coloring of $v_1$ and $v_i$ that
extends to $G[v_1,\ldots,v_i]$ we store at $v_i$ a pointer
to one coloring of $v_{i-1}$ such that these colorings of $v_1, v_{i-1}, v_i$
extend to the path induced by $v_1,\ldots, v_i$.  Now, given any $m$-fold
$L$-colorings of $v_1$ and $v_k$ that extend to $G[v_1,\ldots,v_k]$, it is
straightforward to recursively construct such an extension.
\end{proof}

%
%
%
%

\section*{Acknowledgments}
This paper would not exist if not for Conjecture~\ref{main-conj},
posed by Meng, Puleo, and Zhu.  At the author's request, they included
their code for checking $(4,2)$-choosability as an auxiliary file
for~\cite{MPZarxiv}.  Also, most of the ``bad'' list assignments we
use here already appear in their paper.  Thanks to the National
Security Agency for supporting this work, under NSA Grant H98230-16-0351.  
Thanks to Howard Community College for its hospitality during the
preparation of this paper.  Thanks to an anonymous referee for helpful feedback.
Finally, thanks to Landon Rabern for his
computer program to verify that the 2-connected graph in
Figure~\ref{figGG} is (4,2)-choosable.

\scriptsize
\bibliographystyle{plain}
\bibliography{GraphColoring1}


\setlength{\tabcolsep}{2pt}
\section*{Appendix: Flat 4-assignments for $K_{2,3}$ and $C_4$}
We denote $V(K_{2,3})$ by $\{x_1,x_2,x_3,y_1,y_2\}$, where $d(x_i)=2$ and
$d(y_j)=3$.
We denote $V(C_4)$ by $\{v_1,v_2,v_3,v_4\}$ in order around the cycle.
The double lines in each table demark 4-assignments with different pot sizes, and the single lines
demark different multisets of sizes of $|V(G_\alpha)|$, over all $\alpha\in
\pot(L)$.

\begin{table}[!h]
\centering
\footnotesize
\begin{minipage}[t]{.33\linewidth}
\begin{tabular}{ccccc}
$L(x_1)$ & $L(x_2)$ & $L(x_3)$ & $L(y_1)$ & $L(y_2)$\\
\hline
1234& 2356& 4567& 1567& 2347\\
\hline
\hline
3456& 1256& 1234& 1256& 3456\\
1356& 2456& 1234& 1256& 3456\\
3456& 1235& 1246& 1256& 3456\\
3456& 1235& 1246& 1356& 2456\\

2456& 3456& 1236& 1235& 1456\\
2456& 2356& 1346& 1235& 1456\\
1456& 2356& 2346& 1235& 1456\\
1256& 3456& 2346& 1235& 1456\\ 

1256& 3456& 3456& 1256& 1234\\ 
1356& 2456& 3456& 1256& 1234\\
1256& 3456& 3456& 1235& 1246\\
1356& 2456& 3456& 1235& 1246\\

\hline

1236& 1246& 3456& 1256& 3456\\
1236& 1246& 3456& 1356& 2456\\

\hline

1236& 2456& 3456& 2356& 1456\\ 
1256& 2356& 3456& 2456& 1346\\

\hline

1456& 2456& 3456& 3456& 1236\\

\hline
\end{tabular}
\end{minipage}
%
\begin{minipage}[t]{.33\linewidth}
\begin{tabular}{ccccc}
$L(x_1)$ & $L(x_2)$ & $L(x_3)$ & $L(y_1)$ & $L(y_2)$\\

\hline

1234& 2456& 3456& 2356& 1456\\ 
1245& 2456& 3456& 2356& 1346\\
1456& 2456& 3456& 2356& 1234\\ 
1245& 3456& 2356& 2456& 1346\\ 
1456& 3456& 2356& 2456& 1234\\ 
1245& 3456& 3456& 2456& 1236\\
1245& 2456& 3456& 3456& 1236\\ 
1234& 2456& 3456& 3456& 1256\\ 

\hline

1234& 3456& 3456& 2356& 1456\\ 
1356& 2456& 3456& 3456& 1234\\ 

\hline
\hline
1234& 1235& 1245& 1345& 2345\\
\hline
1235& 2345& 2345& 1245& 1345\\
\hline
1235& 1245& 2345& 1345& 2345\\
\hline
1245& 2345& 2345& 1345& 2345\\
\hline
1345& 1345& 2345& 1245& 2345\\
\hline
1245& 1345& 2345& 1345& 2345\\
\hline
\hline
1234& 1234& 1234& 1234& 1234\\
\hline
\\
\end{tabular}
\end{minipage}
\begin{minipage}[t]{.22\linewidth}
\begin{tabular}{cccc}
$L(v_1)$ & $L(v_2)$ & $L(v_3)$ & $L(v_4)$\\
\hline
1256& 1345& 3456& 2346\\
\hline
\hline
1235& 1245& 1345& 2345\\
1245& 1345& 2345& 2345\\
\hline
\hline
1234& 1234& 1234& 1234\\
\hline
\\
\\
\\
\\
\\
\\
\\
\\
\\
\\
\\
\\
\\
\\
\end{tabular}
\end{minipage}
\caption{The 35 flat 4-assignments for $K_{2,3}$ and 4 flat 4-assignments for
$C_4$.}
\end{table}

\end{document}